\documentclass{article}

\usepackage[latin1]{inputenc}
\usepackage[T1]{fontenc}
\usepackage{amssymb,amsmath,amsthm,amscd,mathrsfs}
\usepackage[all]{xy}
\usepackage{color}

\newtheorem{thm}{Theorem}[section]
\newtheorem{defi}[thm]{Definition}
\newtheorem{prop}[thm]{Proposition}
\newtheorem{lemme}[thm]{Lemma}
\newtheorem{cor}[thm]{Corollary}

\DeclareMathOperator{\Z}{\mathbb{Z}}
\DeclareMathOperator{\sing}{Sing}
\DeclareMathOperator{\codim}{Codim}
\DeclareMathOperator{\Hom}{Hom}
\DeclareMathOperator{\sym}{Sym}

\DeclareMathOperator{\rk}{rk}

\DeclareMathOperator{\Ext}{Ext}

\DeclareMathOperator{\discr}{discr}
\DeclareMathOperator{\Fix}{Fix}
\DeclareMathOperator{\Supp}{Supp}
\DeclareMathOperator{\End}{End}
\DeclareMathOperator{\Vect}{Vect}
\DeclareMathOperator{\tors}{torsion}
\DeclareMathOperator{\id}{id}
\DeclareMathOperator{\rktor}{rktor}

\newcommand{\eq}[1][r]
{\ar@<-3pt>@{-}[#1]
\ar@<-1pt>@{}[#1]|<{}="gauche"
\ar@<+0pt>@{}[#1]|-{}="milieu"
\ar@<+1pt>@{}[#1]|>{}="droite"
\ar@/^2pt/@{-}"gauche";"milieu"
\ar@/_2pt/@{-}"milieu";"droite"}

\newcommand{\incl}[1][r]
  {\ar@<-0.2pc>@{^(-}[#1] \ar@<+0.2pc>@{-}[#1]}

\begin{document}

\title{\bf Beauville--Bogomolov lattice for a singular symplectic variety of
dimension 4}

\author{Grégoire \textsc{Menet}} 

\maketitle

\begin{abstract}
The Beauville--Bogomolov lattice is computed for a simplest singular symplectic manifold of dimension 4, obtained as a partial desingularization of the quotient $S^{[2]}/\iota$, where $S^{[2]}$ is the Hilbert square of a K3 surface $S$ and $\iota$ is a symplectic involution on it. This result applies, in particular, to the singular symplectic manifolds of dimension 4, constructed by Markushevich--Tikhomirov as compactifications of families of Prym varieties of a linear system of curves on a K3 surface with an anti-symplectic involution.
\end{abstract}

\section*{Introduction}
The irreducible symplectic manifolds are important objects in algebraic geometry, in particular because of their role in the Bogomolov decomposition Theorem \cite{Bogo}. 
Irreducible symplectic manifolds are defined as compact holomorphically symplectic Kähler manifolds with trivial fundamental group, whose symplectic structure is unique up to proportionality.

For any irreducible symplectic manifold $Z$, the group $H^{2}(Z,\Z)$ can be endowed with a deformation invariant integral primitive bilinear form $B_{Z}$ which encodes an important topological information. This form is called the Beauville--Bogomolov form. Another invariant related to the Beauville--Bogomolov form is the Fujiki constant $C_{Z}$, which is a positive rational number.

However, very few deformation classes of manifolds of this type are known. For the moment, we just know the examples of Beauville (see\cite{Beauville}) and O'Grady (see \cite{Grady1} and \cite{Grady2}).
The Beauville--Bogomolov form and the Fujiki constant of these examples are calculated in \cite{Beauville} for Beauville's examples and in \cite{Rapa1} and \cite{Rapa2} for the O'Grady varieties.

Find new examples of such manifolds is a very hard problem. One idea to get round this difficulty is to turn back to the setting of Fujiki (see \cite{Fujiki}), who considered symplectic V-manifolds. A V-manifold is an algebraic variety with at worst finite quotient singularities. A V-manifold will be called symplectic if its nonsingular locus is endowed with an everywhere nondegenerate holomorphic 2-form. A symplectic V-manifold will be called irreducible if it is complete, simply connected, and if the holomorphic 2-form is unique up to $\mathbb{C}^*$. The examples of irreducible symplectic V-manifolds given by Fujiki are all, up to deformations of complex structure, partial resolutions of finite quotients of the products of two symplectic surfaces.
In \cite{Markou} Markushevich and Tikhomirov provide a new construction of an irreducible symplectic V-manifold $\mathcal{P}$ of dimension 4, 
obtained as a compactification of a family of Prym varieties of a linear system of curves on a K3 surface.

According to Namikawa in \cite{Nanikawa}, it is possible to endow the second cohomology group of some irreducible symplectic V-manifolds with a Beauville--Bogomolov form. It is a natural question to calculate the Beauville--Bogomolov form of the V-manifold $\mathcal{P}$ of Markushevich and Tikhomirov.

According to Corollary 5.7 of \cite{Markou}, this variety is related by a flop to a partial resolution $M'$ of the quotient of the Hilbert square $S^{[2]}$ of some K3 surface $S$ by a symplectic involution. 
Mongardi's result (Theorem 1.3 in \cite{Mongardi}) says that an irreducible symplectic manifold of $K3^{[2]}$-type endowed with a symplectic involution can be deformed to a couple $(S^{[2]}, i^{[2]})$ where $S$ is a K3 surface, $i$ is a symplectic involution on $S$ and $i^{[2]}$ is the involution induced by $i$ on $S^{[2]}$.
Thus, studying a partial resolution of a quotient $S^{[2]}/i^{[2]}$,
we will give the Beauville--Bogomolov form of a partial resolution of all irreducible symplectic manifolds of $K3^{[2]}$-type quotiented by a symplectic involution (Theorem \ref{theorem}). This allows us to solve the initial problem on the Beauville--Bogomolov form of $\mathcal{P}$.

The proof of Theorem \ref{theorem} needs a very good understanding of the second cohomology group of varieties $M'$. With this goal, we give examples of the use of relevant tools from lattice theory, equivariant cohomology and Smith theory for handling cohomology of quotient varieties. This could be useful for many other examples.

In the first section of this article, we will consider an easier problem as a testing ground for our techniques. We will study the second cohomology group of the quotient of a K3 surface by a Nikulin involution, endowed with the cup product. In the second section, we will prove Theorem \ref{theorem}. Finally we will apply this result to the variety of Markushevich and Tikhomirov in the third section.
\newline

\textbf{Acknowledgement.} I would like to thank Dimitri Markushevich for his help.
 
\section{\bf Cohomology of the quotient of a K3 surface by a Nikulin involution}
\begin{prop}\label{Nikuinvo}
Let $X$ be a K3 surface and $i$ a Nikulin involution (i.e. a symplectic involution) on $X$. We denote $\overline{Y}=X/i$. Then $H^{2}(\overline{Y},\Z)$ endowed with the cup product is isometric to $E_{8}(-1)\oplus U(2)^{3}$.
\end{prop}
\subsection{Pullback and pushforward via the quotient map}
We know very well the lattices $H^{2}(X,\Z)$ and $H^{2}(X,\Z)^{i}$ (all the groups are endowed with the cup product).
We can find in \cite{Alessandra} the following Proposition (it is a consequence of the proof of Theorem 5.7 of \cite{Mo}).
\begin{prop}\label{Morri}
There is an isometry $H^{2}(X,\Z)\cong U^{3}\oplus E_{8}(-1)\oplus E_{8}(-1)$
such that $i^{*}$ acts as follows: $$i^{*}:H^{2}(X,\Z)\cong U^{3}\oplus E_{8}(-1)\oplus E_{8}(-1)\rightarrow H^{2}(X,\Z), (u,x,y)\longmapsto(u,y,x).\ \ \ \ (1)$$
This implies that 
the invariant sublattice is
$$H^{2}(X,\Z)^{i}\cong\left\{(u,x,x)\in U^{3}\oplus E_{8}(-1)\oplus E_{8}(-1)\right\}\cong U^{3}\oplus E_{8}(-2).\ \ \ \ \ \ \ (*)$$
The anti-invariant sublattice is the orthogonal complement to the invariant sublattice.
$$(H^{2}(X,\Z)^{i})^{\bot}\cong\left\{(0,x,-x)\in U^{3}\oplus E_{8}(-1)\oplus E_{8}(-1)\right\}\cong E_{8}(-2).$$
\end{prop}

We denote by $\pi:X\rightarrow\overline{Y}$ the quotient map, and we will study the natural morphism:
$$\pi^{*}:H^{2}(\overline{Y},\Z)\rightarrow H^{2}(X,\Z)^{i}.$$
A fundamental tool for this study is given by the following proposition, which follows from \cite{Smith}.
\begin{prop}\label{Smithy}
Let $G$ be a finite group of order $d$ acting on a variety $X$ with the orbit map $\pi:X\rightarrow X/G$ which is a $d$-fold ramified covering. Then 
we can construct a homomorphism $\pi_{*}:H^{*}(X,\Z)\rightarrow H^{*}(X/G,\Z)$ such that
$$\pi_{*}\circ\pi^{*}=d\id_{H^{*}(X/G,\Z)}, \ \ \ \ \ \ \pi^{*}\circ\pi_{*}=\sum_{g\in G}{g^{*}}.$$
\end{prop}
It easily implies the corollary:
\begin{cor}\label{piso}
There is a morphism $\pi_{*}:H^{2}(X,\Z)\rightarrow H^{2}(\overline{Y},\Z)$ with 
$$\pi_{*}\circ\pi^{*}=2\id_{H^{2}(\overline{Y},\Z)}, \ \ \ \ \ \ \pi^{*}\circ\pi_{*}=\id_{H^{2}(X,\Z)}+i^{*}.\ \ \ \ \ (**)$$
Moreover,
$$\pi_{*|H^{2}(X,\Z)^{i}}\circ\pi^{*}=2\id_{H^{2}(\overline{Y},\Z)}, \ \ \ \ \ \ \pi^{*}\circ\pi_{*|H^{2}(X,\Z)^{i}}=2\id_{H^{2}(X,\Z)^{i}}.$$
\end{cor}
So we want to deduce some information on $H^{2}(\overline{Y},\Z)$ from that on $H^{2}(X,$ $\Z)^{i}$.
If we tensor $\pi^{*}$ by $\mathbb{Q}$, then it is an isomorphism by Corollary \ref{piso}.
Moreover $\overline{Y}$ is simply connected by Lemma 1.2 of \cite{Fujiki}, so $\pi^{*}$ is an injection.
It remains to study the problem of the surjectivity.
The restriction  $\pi_{*|H^{2}(X,\Z)^{i}}$ is injective, and we need to understand the image of $H^{2}(X,\Z)^{i}$ by $\pi_{*}$. 
The main question is whether an element $\pi_{*}(x)$ with $x\in H^{2}(X,\Z)^{i}$ is divisible by $2$ in $H^{2}(\overline{Y},\Z)$ or not.

For simplicity, we will denote \textbf{the cup product by a dot}.
The following lemma is a consequence of Proposition \ref{Smithy}.
\begin{lemme}\label{dot}
Let $x$ and $y$ be in $H^{2}(X,\Z)^{i}$. Then $\pi_{*}(x)\cdot\pi_{*}(y)=2x\cdot y$.
\end{lemme}
\begin{proof}
$\pi_{*}(x)\cdot\pi_{*}(y)=\frac{1}{2}\pi^{*}(\pi_{*}(x)\cdot\pi_{*}(y))=\frac{1}{2}\pi^{*}(\pi_{*}(x))\cdot\pi^{*}(\pi_{*}(y))=2x\cdot y$.
\end{proof}

Let $H^{2}(X,\Z)^{i}\cong U^{3}\oplus E_{8}(-2)$ be an isometry from Proposition \ref{Morri} (*), we have $\pi_{*}(H^{2}(X,\Z)^{i})\cong U^{3}(2)\oplus E_{8}(-4)$.
Moreover, the elements $\pi_{*}(x)$ with $x\in E_{8}(-2)\subset H^{2}(X,\Z)^{i}$ are divisible by $2$ in $H^{2}(\overline{Y},\Z)$.
Indeed, we can write $x=t+i^{*}(t)$ by (*), then $\pi_{*}(x)=2\pi_{*}(t)$.
But what can we say about the divisibility of the elements of $\pi_{*}(U^{3})\simeq U^{3}(2)\subset H^{2}(\overline{Y},\Z)$? This is the main difficulty.

To avoid confusion, we fix an isometry $H^{2}(X,\Z)^{i}\cong U^{3}\oplus E_{8}(-2)$ from Proposition \ref{Morri}, and we will \textbf{write $U^{3}\oplus E_{8}(-2)$ instead of $H^{2}(X,\Z)^{i}$} throughout this section. 
\subsection{Lifting to the resolution of singularities of the quotient}
The variety $\overline{Y}$ is singular, and we want to relate $H^{2}(\overline{Y},\Z)$ to the cohomology of the blowup in the singular locus.
By \cite{Ni1}, Section 5, we know that $i$ has exactly eight fixed points.
Let $\widetilde{\varphi}:\widetilde{X}\rightarrow X$ be the blow-up of $X$ in the eight fixed points of $i$. We denote by $\widetilde{i}$ the involution on $\widetilde{X}$ induced by $i$. Let $\varphi: Y\rightarrow\overline{Y}$ be the blow-up of $\overline{Y}$ in its eight singular points. It is a K3 surface. We have $Y\cong\widetilde{X}/\widetilde{i}$. We get the following commutative diagram:
$$\xymatrix{
 X\ar@(dl,ul)[]^{i} \ar[r]^{\pi}& \overline{Y} & & (2)\\
 \widetilde{X} \ar@(dl,ul)[]^{\widetilde{i}} \ar[u]^{\widetilde{\varphi}} \ar[r]^{\widetilde{\pi}} & Y. \ar[u]_{\varphi} & & }$$
The study of the lattices $H^{2}(Y,\Z)$ and $H^{2}(\widetilde{X},\Z)$ will give the required information on $\pi_{*}(U^{3})\subset H^{2}(\overline{Y},\Z)$. 
We have the following diagram of cohomology groups:

$$\xymatrix{
H^{2}(\overline{Y},\Z)\ar@/_/[r]_{\pi^{*}} \ar[d]_{\varphi^{*}}& H^{2}(X,\Z)\ar@/_/[l]^{\pi_{*}}\ar[d]^{\widetilde{\varphi}^{*}} & (3)\\
H^{2}(Y,\Z)\ar@/^/[r]^{\widetilde{\pi}^{*}}& \ar@/^/[l]_{\widetilde{\pi}_{*}} H^{2}(\widetilde{X},\Z). & 
}$$

Denote $E_{k}$, $k=1,...,8$ the exceptional divisors in $\widetilde{X}$ over the fixed points of $i$ in $X$ and $N_{k}=\widetilde{\pi}(E_{k})$ their images in $Y$; these are $(-2)$-curves.
Then we have:
$$H^{2}(\widetilde{X},\Z)\cong H^{2}(X,\Z)\oplus(\oplus_{i=1}^{8}\Z E_{i})\cong U^{3}\oplus E_{8}(-1)^{2}\oplus \left\langle -1\right\rangle^{8},$$

$$\widetilde{\varphi}^{*}:H^{2}(X,\Z)\rightarrow H^{2}(\widetilde{X},\Z)= H^{2}(X,\Z)\oplus(\oplus_{i=1}^{8}\Z E_{i}),\ \ \ \ \ x\longmapsto(x,0),$$
and $$\widetilde{\varphi}_{*}: H^{2}(\widetilde{X},\Z)\rightarrow H^{2}(X,\Z), \ \ \ \ \ (x,e)\rightarrow x.$$

Since Lemma \ref{dot} is also true for $\widetilde{\pi}_{*}$, the lattice $\widetilde{\pi}_{*}(\widetilde{\varphi}^{*}(U^{3}))\subset H^{2}(Y,\Z)$ is isometric to $U^{3}(2)$. We will show that this lattice is primitive in $H^{2}(Y,\Z)$ and then deduce that $\pi_{*}(U^{3})\subset H^{2}(\overline{Y},\Z)$ is also primitive. To this end, we will use the fact that the lattice $H^{2}(Y,\Z)$ is unimodular. 

We have $\widetilde{\varphi}\circ\widetilde{i}=i\circ\widetilde{\varphi}$, so:

$$\xymatrix@R=0pt{
\widetilde{i}^{*}: H^{2}(\widetilde{X},\Z)\eq[r] & U^{3}\oplus E_{8}(-1)\oplus E_{8}(-1)\oplus \left\langle -1\right\rangle^{8}\ar[r]& H^{2}(\widetilde{X},\Z)\\ 
 & (u,x,y,n)\ar[r]& (u,y,x,n).
}$$

Then we have $H^{2}(\widetilde{X},\Z)^{\widetilde{i}}\cong U^{3}\oplus E_{8}(-2)\oplus \left\langle -1\right\rangle^{8}$. Corollary \ref{piso} and Lemma \ref{dot} are also true for $Y$ and $\widetilde{X}$, hence $\widetilde{\pi}_{*}(H^{2}(\widetilde{X},\Z)^{\widetilde{i}})\cong U^{3}(2)\oplus E_{8}(-4)\oplus\left\langle -2\right\rangle^{8}$.
And as we remarked before, the elements $\widetilde{\pi}_{*}(x)$ with $x\in E_{8}(-2)\subset H^{2}(\widetilde{X},\Z)^{\widetilde{i}}$ are divisible by $2$ in $H^{2}(Y,\Z)$.
Then, we have in $H^{2}(Y,\Z)$ the sublattice $\widetilde{\pi}_{*}(\widetilde{\varphi}^{*}(U^{3}))\oplus\frac{1}{2}\widetilde{\pi}_{*}(\widetilde{\varphi}^{*}(E_{8}(-2)))\oplus \widetilde{\pi}_{*}(\left\langle -1\right\rangle^{8})\simeq U^{3}(2)\oplus E_{8}(-1)\oplus\left\langle -2\right\rangle^{8}$.
Now, we can state the following two lemmas.
\begin{lemme}\label{primitif}
The minimal primitive sublattice of $H^{2}(Y,\Z)$ containing the $N_{i}$ is 
$$N=\left\langle N_{1},...,N_{8},\hat{N}\right\rangle,$$
where $$\hat{N}:=\frac{N_{1}+...+N_{8}}{2}.$$
It is called the \textit{Nikulin lattice}.
\end{lemme}
\begin{proof}
See \cite{Mo} section 5. It is not possible to generalize this proof to the case of the involution on the Hilbert scheme. We will give another proof such a generalization in Section \ref{OtherP}.
\end{proof}
\begin{lemme}\label{U2}
The sublattice $\widetilde{\pi}_{*}(\widetilde{\varphi}^{*}(U^{3}))\subset H^{2}(Y,\Z)$ is primitive.
\end{lemme}
\begin{proof}
We denote $S=N\oplus \frac{1}{2}\widetilde{\pi}_{*}(\widetilde{\varphi}^{*}(E_{8}(-2)))\subset H^{2}(Y,\Z)$. Then $S^{\bot}$ is the minimal primitive overlattice of $\widetilde{\pi}_{*}(\widetilde{\varphi}^{*}(U^{3}))\simeq U^{3}(2)$.
Since $H^{2}(Y,\Z)$ is unimodular, we have $A_{S}\simeq A_{S^{\bot}}$. 
In particular, we have $|\discr S|=|\discr S^{\bot}|$.
But $|\discr S|=|\discr N|=\frac{2^{8}}{4}=2^{6}$ and $|\discr U^{3}(2)|=2^{6}$.
Therefore $S^{\bot}=\widetilde{\pi}_{*}(\widetilde{\varphi}^{*}(U^{3}))$, and $\widetilde{\pi}_{*}(\widetilde{\varphi}^{*}(U^{3}))$ is primitive in $H^{2}(Y,\Z)$.

See also Lemma 1.10 of \cite{Alessandra}.
\end{proof}

\subsection{Proof of Proposition \ref{Nikuinvo}}\label{proof}
Now it remains to show how we deduce the primitivity of $\pi_{*}(U^{3})$ in $H^{2}(\overline{Y},\Z)$ from the primitivity of 
$\widetilde{\pi}_{*}(\widetilde{\varphi}^{*}(U^{3}))$ in $H^{2}(Y,\Z)$.
We complete the diagram (3) as follows:

$$\xymatrix{S \incl[d] & U^{3} \incl[d]\\
H^{2}(\overline{Y},\Z)\ar@/_/[r]_{\pi^{*}} \ar[d]_{\varphi^{*}}& H^{2}(X,\Z)\ar@/_/[l]^{\pi_{*}}\ar[d]^{\widetilde{\varphi}^{*}}\\
H^{2}(Y,\Z)\ar@/^/[r]^{\widetilde{\pi}^{*}}& \ar@/^/[l]_{\widetilde{\pi}_{*}} H^{2}(\widetilde{X},\Z)\\
\widetilde{\pi}_{*}(\widetilde{\varphi}^{*}(U^{3}))\incl[u]& \widetilde{\varphi}^{*}(U^{3}).\incl[u]}$$
Here $\mathcal{S}$ is the minimal primitive overlattice of $\pi_{*}(U^{3})$ in $H^{2}(\overline{Y},\Z)$, and we know that $\widetilde{\pi}_{*}(\widetilde{\varphi}^{*}(U^{3}))$ is primitive in $H^{2}(Y,\Z)$.
We are going to show that $S=\pi_{*}(U^{3})$.

The proof use the three ingredients: 
\begin{itemize}
\item the primitivity of $\widetilde{\pi}_{*}(\widetilde{\varphi}^{*}(U^{3}))$,
\item the knowledge of the map:
$$\widetilde{\varphi}^{*}:H^{2}(X,\Z)\rightarrow H^{2}(\widetilde{X},\Z)= H^{2}(X,\Z)\oplus(\oplus_{i=1}^{8}\Z E_{i}),\ \ \ \ \ x\longmapsto(x,0),$$
\item the commutativity of the last diagram.
\end{itemize}
What follows is a formal proof.

Let $x\in S$, $\pi^{*}(x)\in U^{3}\subset H^{2}(X,\Z)$ and $\widetilde{\varphi}^{*}(\pi^{*}(x))\in \widetilde{\varphi}^{*}(U^{3})$.
Then by commutativity, $\widetilde{\pi}^{*}(\varphi^{*}(x))\in \widetilde{\varphi}^{*}(U^{3})$, so $2\varphi^{*}(x)\in\widetilde{\pi}_{*}(\widetilde{\varphi}^{*}(U^{3}))$. Since $\widetilde{\pi}_{*}(\widetilde{\varphi}^{*}(U^{3}))$ is primitive in  $H^{2}(Y,\Z)$, we have $\varphi^{*}(x)\in\widetilde{\pi}_{*}(\widetilde{\varphi}^{*}(U^{3}))$. It means that $\varphi^{*}(x)=\widetilde{\pi}_{*}(\widetilde{\varphi}^{*}(u))$, where $u\in U^{3}$.
Then we have $\widetilde{\pi}^{*}(\varphi^{*}(x))=2\widetilde{\varphi}^{*}(u)$, and by commutativity, we get $\widetilde{\varphi}^{*}(\pi^{*}(x))=2\widetilde{\varphi}^{*}(u)$. Therefore $\pi^{*}(x)=2u$ and $2x=\pi_{*}(\pi^{*}(x))=2\pi_{*}(u)$, so $x=\pi_{*}(u)\in \pi_{*}(U^{3})$.

\subsection{Another proof of Lemma \ref{primitif}}\label{OtherP}
In this section, we give a different proof of Lemma \ref{primitif}, which could be generalized to many other examples.
Let $U=Y\setminus \cup_{i=1}^{8}N_{i}$.  We will use the equivariant cohomology to compute $H^{2}(U,\Z)$.
Before this calculation, we give the link between $H^{2}(U,\Z)$ and $H^{2}(Y,\Z)$.
\subsubsection{The groups $H^{2}(U,\Z)$ and $H^{2}(Y,\Z)$}\label{UY}
\begin{prop}\label{HU0}
We have:
$$H^{2}(U,\Z)=H^{2}(Y,\Z)/\oplus_{i=1}^{8}\Z N_{i}.$$
\end{prop}
\begin{proof}
Look at the following exact sequence:
$$\xymatrix@C=20pt{H^{1}(U,\Z)\ar[r]& H^{2}(Y,U,\Z)\ar[r] & H^{2}(Y,\Z)\ar[r] & H^{2}(U,\Z)\ar[r] & H^{3}(Y,U,\Z)}$$

Since $Y$ is smooth, by Thom's isomorphism (see Section 11.1.2 of \cite{Voisin}), $$H^{i}(Y,U,\Z)\cong H^{i-2}(\cup_{i=1}^{8}N_{i},\Z).$$
Then $H^{2}(Y,U,\Z)=\oplus_{i=1}^{8}H^{0}(N_{i},\Z)=\Z^{8}$ and $H^{3}(Y,U,\Z)=\oplus_{i=1}^{8}H^{3}(N_{i},\Z)=0$.
Moreover we have the universal coefficient theorem for a variety $V$.
$$\xymatrix{0\ar[r] & \Ext(H_{n-1}(V),\Z)\ar[r] & H^{n}(V,\Z)\ar[r] & \Hom(H_{n}(V),\Z)\ar[r]& 0.}$$
Since $U\simeq\overline{Y}\setminus \sing Y$, the last exact sequence and Lemma 1.6 of \cite{Fujiki} give $H^{1}(U,\Z)=0$.
Then we get $$H^{2}(U,\Z)=H^{2}(Y,\Z)/\oplus_{i=1}^{8}\Z N_{i}.$$
\end{proof}

Consequently we can deduce the behaviour of the $N_{i}$ in $H^{2}(Y,\Z)$ from the cohomology group $H^{2}(U,\Z)$.
\subsubsection{Reminder on equivariant cohomology}

Let $M$ be a variety and $G$ a group acting on $M$. 
Let $EG\rightarrow BG$ be a universal $G$-bundle in the category of CW-complexes. 
Denote by $M_{G}=EG\times_{G}M$ the orbit space for the diagonal action of $G$ on the product $EG\times M$ 
and $f: M_{G}\rightarrow BG$ the map induced by the projection onto the first factor. 
The map $f$ is a locally trivial fibre bundle with typical fibre $M$ and structure group $G$.
We define $H_{G}^{*}(M,\Z):=H^{*}(EG\times_{G}M,\Z)$ the G-equivariant cohomology groups of $M$.
We recall that when $G$ acts freely, we have an isomorphism:
$$H^{*}(M/G,\Z)\simeq H_{G}^{*}(M,\Z),$$
induced by the natural map $\sigma : EG\times_{G}M\rightarrow M/G$,
see for instance \cite{Bott}.
Moreover, the Leray-Serre spectral sequence associated to the map $f$ gives a spectral sequence converging to the equivariant cohomology:
$$E_{2}^{p,q}:=H^{p}(G;H^{q}(M,\Z))\Rightarrow H_{G}^{p+q}(M,\Z).$$

\subsubsection{Application}
\begin{prop}
We have:
$$H^{2}(U,\Z)\simeq H^{2}(X,\Z)^{i}\oplus \Z/2\Z.$$
\end{prop}
\begin{proof}
In our case $G=I:=\left\{\id, i\right\}$ and $M=V:=X\setminus \Fix i$.
Then, we have $$H^{2}(U,\Z)\simeq H_{I}^{2}(V,\Z)$$
and
$$E_{2}^{p,q}:=H^{p}(I;H^{q}(V,\Z))\Rightarrow H_{I}^{p+q}(V,\Z).$$
By using the following resolution by free $\Z[I]$-modules
$$\xymatrix@C=10pt{...\ar[r]^{i-1}& \Z[I]\ar[r]^{i+1} & \Z[I]\ar[r]^{i-1} & \Z[I]\ar[r] & \Z},$$
we get $$H^{0}(I;H^{q}(V,\Z))=H^{q}(V,\Z)^{i},$$
\begin{align*}
&H^{2k+1}(I;H^{q}(V,\Z))\\
&=\left\{\left. m\in H^{q}(V,\Z)\right|i^{*}(m)=-m\right\}/\left\{\left. i^{*}(m)-m\right|m\in H^{q}(V,\Z) \right\},\\
&H^{2k+2}(I;H^{q}(V,\Z))\\
&=\left\{\left. m\in H^{q}(V,\Z)\right|i^{*}(m)=m\right\}/\left\{\left. i^{*}(m)+m\right|m\in H^{q}(V,\Z) \right\},
\end{align*}
for all $k\in\mathbb{N}$. Then by Proposition \ref{Morri} we have $H^{1}(I,H^{2}(V,\Z))=0$.
The unique difficulty remaining for the calculation is to see that the differential of the page $E_{3}$,
$d_{3}:H^{2}(I;H^{0}(V,\Z))\rightarrow H^{0}(I;H^{3}(V,\Z))$, is trivial. 
Consider the exact sequence:
$$\xymatrix@C=20pt{H^{3}(X,\Z)\ar[r]& H^{3}(V,\Z)\ar[r] & H^{4}(X,V,\Z)}.$$
Since $H^{3}(X,\Z)=0$ and $H^{4}(X,V,\Z)\simeq \Z^{8}$, the group $H^{3}(V,\Z)$ is torsion free.
It follows from $H^{2}(I;H^{0}(V,\Z))=\Z/2\Z$ and $H^{0}(I;H^{3}(V,\Z))=H^{3}(V,\Z)^{i}$ that $d_{3}=0$.

Hence the spectral sequence gives 
$$H_{I}^{2}(V,\Z)=H^{0}(I;H^{2}(V,\Z))\oplus H^{1}(I;H^{1}(V,\Z))\oplus H^{2}(I;H^{0}(V,\Z)).$$
It follows that
$$H_{I}^{2}(V,\Z)=H^{2}(V,\Z)^{i}\oplus \Z/2\Z.$$
Moreover $H^{2}(V,\Z)=H^{2}(X,\Z)$.
Indeed, we have the following exact sequence:
$$\xymatrix@C=20pt{ H^{2}(X,V,\Z)\ar[r] & H^{2}(X,\Z)\ar[r] & H^{2}(V,\Z)\ar[r] & H^{3}(X,V,\Z)},$$
with $H^{2}(X,V,\Z)=H^{3}(X,V,\Z)=0$.
It implies
$$H^{2}(U,\Z)\simeq H^{2}(X,\Z)^{i}\oplus \Z/2\Z.$$
\end{proof}

\subsubsection{End of the proof}\label{end1}
We have seen that the torsion of $H^{2}(U,\Z)$ is equal to $\Z/2\Z$. Moreover, we know by Proposition \ref{HU0} that the torsion of $H^{2}(U,\Z)$ is isomorphic to $N/\oplus_{i=1}^{8}\Z N_{i}$, where $N$ is the minimal primitive overlattice of $H^{2}(Y,\Z)$ containing the $N_{i}$.
Then $N=\left\langle N_{1},...,N_{8},\widetilde{N}\right\rangle$, where $\widetilde{N}=\frac{\epsilon_{1}N_{1}+...+\epsilon_{8}N_{8}}{2}$ with $\epsilon_{i}$ equal to $0$ or $1$, not all equal to $0$.
That is to say, there is a unique non trivial element in $N/\oplus_{i=1}^{8}\Z N_{i}$ which corresponds to the torsion $\Z/2\Z$ of $H^{2}(U,\Z)$.
But, we still know that $N_{1}+...+N_{8}$ is divisible by $2$, indeed, we have $\widetilde{\pi}_{*}(\mathcal{O}_{\widetilde{X}})=\mathcal{O}_{Y}\oplus\mathcal{L}$, with $\mathcal{L}^{2}=\mathcal{O}_{Y}(-N_{1}-...-N_{8})$.
Therefore, we get Lemma \ref{primitif}.
\subsection{A general technique}
Proof of Lemma \ref{U2}, Section \ref{proof}, and Section \ref{OtherP} provide a general method to calculate second cohomology lattices of surfaces quotiented by a finite group.
For example, with exactly the same method, we can prove the following proposition.
\begin{prop}
Let $A$ be a complex torus of dimension 2. We denote $\overline{A}=A/-\id$. Then $H^{2}(\overline{A},\Z)$ endowed with the cup product is isometric to $U(2)^{3}$.
\end{prop}
\subsection{Smith theory}
A use of Smith theory will be necessary for the case of the involution on the Hilbert scheme.
So we will give here an example of its use in the case of the involution on a K3 surface.
This will give another proof of Lemma \ref{U2}.
\subsubsection{Reminder on the basic tools of Smith theory}
Let $T$ be a topological space and let $G$ be a group of prime order $p$ acting on $T$. We fix a generator $g$ of $G$.
Let $\tau:=g-1\in \mathbb{F}_{p}[G]$ and $\sigma:=1+g+...+g^{p-1}\in \mathbb{F}_{p}[G]$. We consider the chain complex $C_{*}(T)$ of $T$ with coefficients in $\mathbb{F}_{p}$ and its subcomplexes $\tau^{i}C_{*}(T)$ for $1\leq i\leq p-1$ (we have $\sigma=\tau^{p-1}$). We denote also $X^{G}$ the fixed locus of the action of $G$ on $T$. We can find in  \cite{SmithTh}, Section 7 the following proposition.
\begin{prop}\label{SmithProp}
\begin{itemize}
\item \cite{Bredon}, Theorem 3.1. For $1\leq i\leq p-1$ there is an exact sequence of complexes:
$$\xymatrix@C=20pt{0\ar[r] &\tau^{i}C_{*}(T)\oplus C_{*}(T^{G})\ar[r]^{\ \ \ \ \ \ f}&C_{*}(T) \ar[r]^{\tau^{p-i}}&\tau^{p-i}C_{*}(T) \ar[r]&0
},$$ where $f$ denotes the sum of the inclusions.
\item \cite{Bredon}, p.125. For $1\leq i\leq p-1$ there is an exact sequence of complexes:
$$\xymatrix@C=20pt{0\ar[r] &\sigma C_{*}(T)\ar[r]^{f}&\tau^{i}C_{*}(T) \ar[r]^{\tau}&\tau^{i+1}C_{*}(T) \ar[r]&0
},$$ where $f$ denotes the inclusion.
\item \cite{Bredon}, (3.4) p.124. There is an isomorphism of complexes:
$$\sigma C_{*}(T)\simeq C_{*}(T/G,T^{G}),$$
where $T^{G}$ is identified with its image in $T/G$.
\end{itemize}
\end{prop}
\subsubsection{Applications}\label{ASmith}
Consider the exact sequence (4):
$$\xymatrix@C=10pt@R=0pt{0\ar[r] &H^{2}(Y,\cup_{k=1}^{8}N_{k},\mathbb{F}_{2})\ar[r]&H^{2}(Y,\mathbb{F}_{2}) \ar[r]& H^{2}(\cup_{k=1}^{8}N_{k},\mathbb{F}_{2})\\
\ar[r]&H^{3}(Y,\cup_{k=1}^{8}N_{k},\mathbb{F}_{2})\ar[r]&0. &
}$$
First, we will calculate the vectorial spaces $H^{2}(Y,\cup_{k=1}^{8}N_{k},\mathbb{F}_{2})$ and $H^{3}(Y,\cup_{k=1}^{8}N_{k},$ $\mathbb{F}_{2})$.
By 3) of Proposition \ref{SmithProp}, we have 
$$H^{*}(Y,\cup_{k=1}^{8}N_{k},\mathbb{F}_{2})\simeq H^{*}_{\sigma}(\widetilde{X}),$$
where $H^{*}_{\sigma}(\widetilde{X})$ is the cohomology group of the complex $\sigma C_{*}(\widetilde{X})$.
\begin{lemme}
We have:
$$h^{2}_{\sigma}(\widetilde{X})=15,\ \ \ \ \ h^{3}_{\sigma}(\widetilde{X})=1.$$
\end{lemme}
\begin{proof}
The exact sequence (3) gives us the following equation:

$$h^{2}_{\sigma}(\widetilde{X})-h^{2}(Y,\mathbb{F}_{2})+h^{2}(\cup_{k=1}^{8}N_{k},\mathbb{F}_{2})-h^{3}_{\sigma}(\widetilde{X})=0.$$
As $h^{2}(Y,\mathbb{F}_{2})=22$, $h^{2}(\cup_{k=1}^{8}N_{k},\mathbb{F}_{2})=8$, we obtain
$$h^{2}_{\sigma}(\widetilde{X})-h^{3}_{\sigma}(\widetilde{X})=14,$$
where $h^{*}_{\sigma}(\widetilde{X})$ denote the dimension over $\mathbb{F}_{2}$ of $H^{*}_{\sigma}(\widetilde{X})$.
Moreover by 2) of Proposition \ref{SmithProp}, we have another exact sequence:
$$\xymatrix@C=10pt@R=0pt{0\ar[r] &H^{1}_{\sigma}(\widetilde{X})\ar[r]&H^{2}_{\sigma}(\widetilde{X}) \ar[r]&H^{2}(\widetilde{X},\mathbb{F}_{2}) \ar[r]&H^{2}_{\sigma}(\widetilde{X})\oplus H^{2}(\cup_{k=1}^{8}E_{k},\mathbb{F}_{2})\\
\ar[r]&H^{3}_{\sigma}(\widetilde{X})\ar[r]&0. & &
}$$
By Lemma 7.4 of \cite{SmithTh}, $h^{1}_{\sigma}(\widetilde{X})=h^{0}(\cup_{k=1}^{8}N_{k},\mathbb{F}_{2})-1$.
Then we get the following equation:
\begin{align*}
&h^{0}(\cup_{k=1}^{8}E_{k},\mathbb{F}_{2})-1-h^{2}_{\sigma}(\widetilde{X})+h^{2}(\widetilde{X},\mathbb{F}_{2})\\
&-h^{2}_{\sigma}(\widetilde{X})-h^{2}(\cup_{k=1}^{8}E_{k},\mathbb{F}_{2})+h^{3}_{\sigma}(\widetilde{X})=0,
\end{align*}
That is
$$29-2h^{2}_{\sigma}(\widetilde{X})+h^{3}_{\sigma}(\widetilde{X})=0.$$
Finally
$$h^{2}_{\sigma}(\widetilde{X})=15,\ \ \ \ \ h^{3}_{\sigma}(\widetilde{X})=1.$$
\end{proof}
Now, we come back to the exact sequence (3):
$$\xymatrix@C=10pt{0\ar[r] &H^{2}(Y,\cup_{k=1}^{8}N_{k},\mathbb{F}_{2})\ar[r]&H^{2}(Y,\mathbb{F}_{2}) \ar[r]^{j^{*}\ \ \ \ \ }& H^{2}(\cup_{k=1}^{8}N_{k},\mathbb{F}_{2})
},$$
where $j:\cup_{k=1}^{8}N_{k}\hookrightarrow Y$ is the inclusion.
The key point of the proof of Lemma \ref{U2} is to see that $\dim_{\mathbb{F}_{2}} j^{*}(H^{2}(Y,\mathbb{F}_{2}))=7$.
Since $h^{2}(Y,\cup_{k=1}^{8}N_{k},\mathbb{F}_{2})=h^{2}_{\sigma}(\widetilde{X})=15$, we have $\dim_{\mathbb{F}_{2}} j^{*}(H^{2}(Y,\mathbb{F}_{2}))=22-15=7$.
We can interpret this in terms of the integer cohomology. Consider the map
\begin{align*}
j^{*}_{\Z}:H^{2}(Y,\Z)&\rightarrow \oplus_{k=1}^{8} H^{2}(N_{k},\Z)\\
 u&\rightarrow (u\cdot N_{1},...,u\cdot N_{8}).
\end{align*}
Since there is not torsion, we have:
$$j^{*}=j^{*}_{\Z}\otimes id_{\mathbb{F}_{2}}: H^{2}(Y,\Z)\otimes\mathbb{F}_{2}\rightarrow \oplus_{k=1}^{8} H^{2}(N_{k},\Z)\otimes\mathbb{F}_{2}.$$
This means that we have 7 independent elements such that the intersections with the $N_{k}$, $k\in\left\{1,...,8\right\}$ are not all even.
But, $j^{*}(\widetilde{\pi}_{*}(\widetilde{\varphi}^{*}(U^{3}))\oplus\frac{1}{2}\widetilde{\pi}_{*}(\widetilde{\varphi}^{*}(E_{8}(-2)))\oplus \left\langle N_{1},...,N_{8}\right\rangle)=0$.
Hence, there are 7 more independent elements in $H^{2}(Y,\Z)$. Moreover, these elements must be of the form $\frac{u+n}{2}$
with $u\in\widetilde{\pi}_{*}(\widetilde{\varphi}^{*}(U^{3}))$ and $n\in \left\langle N_{1},...,N_{8}\right\rangle$.
What follows is a formal proof of this fact.

We know that $j^{*}_{\Z}(\frac{N_{1}+...+N_{8}}{2})=(-1,...,-1)\in \oplus_{k=1}^{8} H^{2}(N_{k},\Z)$,
Let $\overline{x_{1}}=(1,...,1)\otimes 1\in \oplus_{k=1}^{8} H^{2}(N_{k},\mathbb{F}_{2})$. We have $\overline{x_{1}}\in j^{*}(H^{2}(Y,\mathbb{F}_{2}))$.
We complete the singleton $(\overline{x_{1}})$ to a base $(\overline{x_{k}}:=x_{k}\otimes 1)_{1\leq k\leq7}$ of $j^{*}(H^{2}(Y,\mathbb{F}_{2}))$.
For all $1\leq k\leq7$, we have $$x_{k}=j^{*}_{\Z}\left(\frac{u_{k}+n_{k}}{2}\right),$$ 
with $u_{k}\in\widetilde{\pi}_{*}(\widetilde{\varphi}^{*}(U^{3}))$ and $n_{k}$ an integer combination of the $N_{l}$, $1\leq l\leq 8$. 
Indeed, if we have $x_{k}=j^{*}_{\Z}(\frac{e_{k}+u_{k}+n_{k}}{2})$ with $e_{k}\in \frac{1}{2}\widetilde{\pi}_{*}(\varphi^{*}(E_{8}(-1)))$,
by taking the image of $\frac{e_{k}+u_{k}+n_{k}}{2}$ by $\pi^{*}$, we see that $e_{k}$ is divisible by 2 and $j^{*}_{\Z}(\frac{e_{k}+u_{k}+n_{k}}{2})=j^{*}_{\Z}(\frac{u_{k}+n_{k}}{2})$.
We choose $u_{1}=0$ and $n_{1}=N_{1}+...+N_{8}$.

Now, let $\mathcal{N}$ be the vector subspace of $\Vect_{\mathbb{F}_{2}}(N_{1},...,N_{8})$ generated by the $n_{k}$, $1\leq k\leq 7$.
We have: $$\dim_{\mathbb{F}_{2}} \mathcal{N}= 7.$$ To show this, we just need to see that $(n_{k})_{1\leq k\leq7}$ is free.
If $\sum_{k=1}^{7}{ \overline{\epsilon_{k}}n_{k}}=0$, then $\sum_{k=1}^{7}{ \epsilon_{k} n_{k}}=2n$; where $n$ is an integer combination of the $N_{k}$ by definition of $\mathcal{N}$.
Then
\begin{align*} \sum_{k=1}^{7}{\epsilon_{k}x_{k}}&=\sum_{k=1}^{7}{\epsilon_{k}j^{*}_{\Z}\left(\frac{u_{k}+n_{k}}{2}\right)}\\
&=j^{*}_{\Z}\left(\frac{\sum_{k=1}^{7}\epsilon_{k}(u_{k}+n_{k})}{2}\right)\\
&=j^{*}_{\Z}\left(\frac{\sum_{k=1}^{7}\epsilon_{k}u_{k}}{2}\right)+j^{*}_{\Z}(n)\\
&=j^{*}_{\Z}(n).
\end{align*}
This implies $\sum_{k=1}^{7}{\overline{\epsilon_{k}}\overline{x_{k}}}=0$, hence all the $\overline{\epsilon_{k}}=0$.

Let $(v_{l,m})_{1\leq l\leq3,1\leq m\leq2}$ be a base of $U^{3}$ and let $\mathcal{U}$ be subspace of $\Vect_{\mathbb{F}_{2}}((\widetilde{\pi}_{*}$ $(\widetilde{\varphi}^{*}(v_{l,m})))_{1\leq l\leq3,1\leq m\leq2})$ generated by the $u_{k}$, $2\leq k\leq 7$.
We will show that $\mathcal{U}=\Vect_{\mathbb{F}_{2}}((\widetilde{\pi}_{*}(\widetilde{\varphi}^{*}(v_{l,m})))_{1\leq l\leq3,1\leq m\leq2})$.
To do this, we just need to show that the family $(u_{k})_{2\leq k\leq 7}$ is free.
If $\sum_{k=2}^{7}{ \overline{\epsilon_{k}}u_{k}}$ $=0$, then $\sum_{k=2}^{7}{ \epsilon_{k} u_{k}}=2u$ with $u$ an integer combination of the $u_{l,m}$ by definition of $\mathcal{U}$.
Then $u=\frac{\sum_{k=2}^{7}{ \epsilon_{k} n_{k}}}{2}$ is integer.
By Lemma \ref{primitif}, there are just two possibilities:
all the $\epsilon_{k}$ are even or $u=\frac{N_{1}+...+N_{8}}{2}$.
In the second case, we get $\sum_{k=2}^{7}{ \overline{\epsilon_{k}} n_{k}}=n_{1}$, which is impossible.

Now, we are able to show that $\widetilde{\pi}_{*}(\widetilde{\varphi}^{*}(U^{3}))$ is primitive in $H^{2}(Y,\Z)$.
Let $\mathfrak{U}$ be the primitive overgroup of $\widetilde{\pi}_{*}(\widetilde{\varphi}^{*}(U^{3}))$ in $H^{2}(Y,\Z)$ and let $x\in \mathfrak{U}$.
Since $\mathcal{U}=\Vect_{\mathbb{F}_{2}}((\widetilde{\pi}_{*}(\widetilde{\varphi}^{*}(v_{l,m})))_{1\leq l\leq3,1\leq m\leq2})$, we can write $x=v+\frac{\sum_{k=2}^{7}{\epsilon_{k}u_{k}}}{2}$, with $v\in \widetilde{\pi}_{*}(\widetilde{\varphi}^{*}(U^{3}))$.
Then the element $x-\sum_{k=2}^{7}{\epsilon_{k}\frac{u_{k}+n_{k}}{2}}-v=\frac{\sum_{k=2}^{7}{\epsilon_{k}n_{k}}}{2}$ is in $H^{2}(Y,\Z)$.
Then by Lemma \ref{primitif}, all the $\epsilon_{k}$ are even or $\sum_{k=2}^{7}{\epsilon_{k}n_{k}}=n_{1}$, which is impossible.

\textbf{Remark}: With the help of Smith theory, we also can find $H^{3}(\overline{Y},\Z)=\Z/2\Z$ and $H^{4}(\overline{Y},\Z)=\Z$.
\section{Beauville--Bogomolov lattice of a partial resolution of the quotient of a $K3^{[2]}$-type manifold by a symplectic involution}
\subsection{Statement of the main theorem}
Now we want to generalize the previous result by replacing a K3 surface $S$ by its Hilbert scheme of two points $S^{[2]}$. We immediately encounter some new difficulties. The first one is that the cup product is not a bilinear form on $H^{2}(S^{[2]},\Z)$. We have to work with the Beauville--Bogomolov form on $H^{2}(S^{[2]},\Z)$ instead, but it is no more unimodular. Yet another difficulty is that if we take a symplectic involution $\sigma$ on $S^{[2]}$, there is no definition of a Beauville--Bogomolov form on $H^{2}(S^{[2]}/\sigma,\Z)$.
However, it is possible to generalize the definition of the Beauville--Bogomolov form to some singular varieties.
In \cite{Nanikawa}, we find the following two definitions.
\begin{defi}
A normal compact Kähler variety $Z$ is said to be symplectic if there is a nondegenerate holomorphic 2-form $\omega$ on the smooth locus $U$ of $Z$ which extends to a regular 2-form $\widetilde{\omega}$ on a desingularization $\widetilde{Z}$ of $Z$.
If, moreover, $\dim H^{1}(Z,\mathcal{O}_{Z})=0$ and $\dim H^{2}(U,\Omega_{U}^{2})=1$, we say that $Z$ is an irreducible symplectic variety.
\end{defi}
\begin{defi}\label{definition}
Let $Z$ be a $2n$-dimensional projective irreducible symplectic variety and $\nu: \widetilde{Z}\rightarrow Z$ a resolution of singularities of $Z$.
Assume that 
\begin{itemize}
\item  The codimension of the singular locus is $\geq 4$.
\item $Z$ has only $\mathbb{Q}$-factorial singularities.
\end{itemize}
We define the quadratic form $q_{Z}$ on $H^{2}(Z,\mathbb{C})$ by
$$q_{Z}(\alpha):=\frac{n}{2}\int_{\widetilde{Z}}(\widetilde{\omega}\overline{\widetilde{\omega}})^{n-1}\widetilde{\alpha}^2+(1-n)\int_{\widetilde{Z}}\widetilde{\omega}^{n-1}\overline{\widetilde{\omega}}^{n}\widetilde{\alpha}\cdot\int_{\widetilde{Z}}\widetilde{\omega}^{n}\overline{\widetilde{\omega}}^{n-1}\widetilde{\alpha},$$
where $\widetilde{\alpha}:=\nu^{*}\alpha$, $\alpha\in H^{2}(Z,\mathbb{C})$ and $\int_{\widetilde{Z}}\widetilde{\omega}^{n}\cdot\overline{\widetilde{\omega}}^{n}=1$.
\end{defi}
The following theorem is proved in \cite{Mat}.
\begin{thm}\label{Form}
Let $Z$ be a projective irreducible symplectic variety of dimension $2n$ with only $\mathbb{Q}$-factorial singularities, and $\codim\sing Z\geq 4$.
There exists a unique indivisible integral symmetric bilinear form $B_{Z}\in S^{2}(H^2(Z,\Z))^{*}$ and a unique positive constant $c_{Z}\in \mathbb{Q}$, such that for any $\alpha\in H^2(Z,\mathbb{C})$,
$$\alpha^{2n}=c_{Z}B_{Z}(\alpha,\alpha)^n. \ \ \ \ \ \ \ \ \ \ \ (1)$$
For $0\neq \omega\in H^{0}(\Omega_{U}^{2})$
$$B_{Z}(\omega+\overline{\omega},\omega+\overline{\omega})>0.\ \ \ \ \ \ \ \ \ \ \ (2)$$
Moreover the signature of $B_{Z}$ is $(3,h^{2}(Z,\mathbb{C})-3)$.

The form $B_{Z}$ is proportional to $q_{Z}$ and is called the Beauville--Bogomolov form of $Z$.
\end{thm}
\begin{proof}
The statement of the theorem in \cite{Mat} does not say that the form is integral, but it follows from Lemma 2.2 of \cite{Mat} using the proof of Theorem 5 a), c) of \cite{Beauville}.
\end{proof}

For other generalizations of the Beauville--Bogomolov form see \cite{Tim}.
\newline

An irreducible symplectic V-manifold is an irreducible symplectic variety with at worst finite quotient singularities, so it has only $\mathbb{Q}$-factorial singularities. Hence, if its singular locus has codimension greater or equal to four, it is endowed with a Beauville--Bogomolov form.

Now we will describe the singularities of our variety $S^{[2]}/\sigma$. By Theorem 4.1 of \cite{Mongardi} the fixed locus of $\sigma$ is the union of 28 points and a K3 surface $\Sigma$. Then the singular locus of $M:=S^{[2]}/\sigma$ is the union of a K3 and 28 points. The singular locus is not of codimension four. We will lift to a partial resolution of singularities,
$M'$ of $M$, obtained by blowing up the image of $\Sigma$. By Section 2.3 and Lemma 1.2 of \cite{Fujiki}, the variety $M'$ is an irreducible symplectic V-manifold which has singular locus of codimension four.

In fact, we can consider a more general case. We have the following theorem of Mongardi from \cite{Mongardi}:
\begin{thm}(\emph{Theorem 1.3 of Mongardi})
Let $X$ be an irreducible symplectic manifold of $K3^{[2]}$-type and $\sigma$ a symplectic involution on $X$.
Then there exists a K3 surface $S$ endowed with a symplectic involution $i$ such that $(X,\sigma)$ and $(S^{[2]},i^{[2]})$ are deformation equivalent.
\end{thm}
So, instead of $S^{[2]}$, we can consider any irreducible symplectic manifold $X$ of $K3^{[2]}$-type with a symplectic involution $\sigma$ (Theorem 4.1 of \cite{Mongardi} is formulated for this case). 
We will prove the following theorem:
\begin{thm}\label{theorem}
Let $X$ be an irreducible symplectic manifold of $K3^{[2]}$-type and $\sigma$ a symplectic involution on $X$.
Let $\Sigma$ be the K3 surface which is in the fixed locus of $\sigma$.
We denote $M=X/\sigma$ and $M'$ the partial resolution of singularities of $M$ obtained by blowing up the image of $\Sigma$.
Then the Beauville--Bogomolov lattice $H^2(M',\Z)$ is isomorphic to $E_{8}(-1)\oplus U(2)^{3}\oplus(-2)^{2}$, and the Fujiki constant is equal to $6$.
\end{thm}
\subsection{Pullback and pushforward via the quotient map}\label{Focus}
The main difficulty is the same as in the first case.
By Theorem 1.3 of \cite{Mongardi}, we can reduce our study to the case where $(X,\sigma)=(S^{[2]},i^{[2]})$ with $S$ a K3 surface an $i$ a symplectic involution on $S$. We are working with this case during all the section; we will denote $\iota:=i^{[2]}$.

To calculate the Beauville--Bogomolov form on $H^{2}(M',\Z)$, we will use our good knowledge of $(H^{2}(S^{[2]},\Z),B_{S^{[2]}})$,
see Part 2 of \cite{Beauville}. 
\begin{prop}\label{Mong}
There is an isometry $H^{2}(S^{[2]},\Z)\cong U^{3}\oplus(-2)\oplus E_{8}(-1)\oplus E_{8}(-1)$ such that $\iota^{*}$ acts as follows: 
\begin{equation}
\begin{split}
\iota^{*}:H^{2}(S^{[2]},\Z)\cong U^{3}\oplus(-2)\oplus E_{8}(-1)\oplus E_{8}(-1)&\rightarrow H^{2}(S^{[2]},\Z)\\ (u,\delta,x,y)&\longmapsto(u,\delta,y,x).
\end{split}
\label{invo}
\end{equation}
The invariant sublattice is
\begin{equation}
\begin{split}
H^{2}(S^{[2]},\Z)^{\iota}&\cong\left\{(u,\delta,x,x)\in U^{3}\oplus(-2)\oplus E_{8}(-1)\oplus E_{8}(-1)\right\}\\
&\cong U^{3}\oplus(-2)\oplus E_{8}(-2).
\end{split}
\label{inva}
\end{equation}
The anti-invariant sublattice, that is the orthogonal complement to the invariant sublattice, is
$$(H^{2}(S^{[2]},\Z)^{\iota})^{\bot}\cong\left\{(0,0,x,-x)\in U^{3}\oplus(-2)\oplus E_{8}(-1)\oplus E_{8}(-1)\right\}\cong E_{8}(-2).$$
\end{prop}
\begin{proof}
We begin by recalling the result of Beauville \cite{Beauville}.
We have 
\begin{equation}
H^{2}(S^{[2]},\Z)=j(H^{2}(S,\Z))\oplus\Z\delta,
\label{j}
\end{equation}
where $\delta$ is half the diagonal of $S^{[2]}$. In our case, $\delta$ is invariant by $\iota$.
We are going to give the definition of $j$.
Denote by $\omega: S^{2}\rightarrow S^{(2)}$ and $\epsilon: S^{[2]}\rightarrow S^{(2)}$ the quotient map and the blowup in the diagonal respectively. Also denote $Pr_{1}$ and $Pr_{2}$ the first and second projections $S^{2}\rightarrow S$.
For $\alpha\in H^{2}(S,\Z)$, we define $j(\alpha)=\epsilon^{*}(\beta)$, where $\beta$ is the element of $H^{2}(S^{(2)},\Z)$ such that $\omega^{*}(\beta)=Pr_{1}^{*}(\alpha)+Pr_{2}^{*}(\alpha)$.
With this construction we have 
\begin{equation}
j\circ i=\iota\circ j.
\label{ij}
\end{equation}
Moreover, by Beauville \cite{Beauville} again, we have 
\begin{equation}
B_{S^{[2]}}(\delta,\delta)=-2,\ \ \ \ \ B_{S^{[2]}}(j(\alpha_{1}),j(\alpha_{2}))=\alpha_{1}\cdot\alpha_{2},
\label{Beau}
\end{equation}
for all $(\alpha_{1},\alpha_{2})\in H^{2}(S,\Z)^{2}$.
Now, we consider the isometry $H^{2}(S,\Z)\cong U^{3}\oplus E_{8}(-1)\oplus E_{8}(-1)$ of Proposition \ref{Morri} with 
$$i^{*}:H^{2}(S,\Z)\cong U^{3}\oplus E_{8}(-1)\oplus E_{8}(-1)\rightarrow H^{2}(S,\Z), (u,x,y)\longmapsto(u,y,x).$$
Then by \eqref{j} and \eqref{Beau} we get an isometry $H^{2}(S^{[2]},\Z)\cong U^{3}\oplus(-2)\oplus E_{8}(-1)\oplus E_{8}(-1)$, and \eqref{ij} 
implies the wanted formula for $\iota^{*}$
\end{proof}
We now turn to the quotient $S^{[2]}/\iota=:M$.
Let $\Sigma$ be the K3 surface in the fixed locus of $\iota$. Consider the partial resolution of singularities $r':M'\rightarrow M$ obtained by blowing up $\overline{\Sigma}:=\pi(\Sigma)$, where $\pi:S^{[2]}\rightarrow M$ is the quotient map. Denote by $\overline{\Sigma}'$ the exceptional divisor. Let $s_{1}:N_{1}\rightarrow S^{[2]}$ be the blowup of $S^{[2]}$ in $\Sigma$; and denote by $\Sigma_{1}$ the exceptional divisor in $N_{1}$. Denote by $\iota_{1}$ the involution on $N_{1}$ induced by $\iota$. We have $M'\simeq N_{1}/\iota_{1}$, and we denote $\pi_{1}:N_{1}\rightarrow M'$ the quotient map. We sum up the notations in the diagram:
$$\xymatrix{
 M' \ar[r]^{r'}& M \\
   N_{1}\ar@(dl,dr)[]_{\iota_{1}} \ar[r]^{s_{1}} \ar[u]^{\pi_{1}} & S^{[2]}.\ar@(dl,dr)[]_{\iota}\ar[u]^{\pi}
   }$$
We have
$$H^{2}(N_{1},\Z)\cong H^{2}(S^{[2]},\Z)\oplus\Z \Sigma_{1},$$

$$s_{1}^{*}:H^{2}(S^{[2]},\Z)\rightarrow H^{2}(N_{1},\Z)= H^{2}(S^{[2]},\Z)\oplus \Z \Sigma_{1},\ \ \ \ \ x\longmapsto(x,0),$$
$$s_{1*}: H^{2}(N_{1},\Z)\rightarrow H^{2}(S^{[2]},\Z), \ \ \ \ \ (x,e)\rightarrow x.$$
As in Section 1, we will study the morphism $\pi_{1}^{*}:H^{2}(M',\Z)\rightarrow H^{2}(N_{1},\Z)$.
Hence, we need a generalization of Lemma \ref{dot} which comes from Proposition \ref{Smithy}.
\begin{lemme}\label{inter}
Let $\mathcal{V}$ be a complex variety of dimension four endowed with an involution $\tau$.
Let $\pi:\mathcal{V}\rightarrow \mathcal{V}/\tau$ be the quotient map.
\begin{itemize}
\item[(1)] 
For $x_{i}\in H^{2}(\mathcal{V},\Z)^{\tau}$, $1\leq i\leq 4$, we have $\pi_{*}(x_{1})\cdot...\cdot\pi_{*}(x_{4})=8x_{1}\cdot...\cdot x_{4}$.
\item[(2)]
Let $x,y\in H^{4}(\mathcal{V},\Z)^{\tau}$. Then $\pi_{*}(x)\cdot \pi_{*}(y)=2x\cdot y$.
\item[(3)]
The map $\pi^{*}$ is injective on $H^{*}(\mathcal{V}/\tau,\Z)/(\tors)$.
If $H^{4}(\mathcal{V},\Z)$ is torsion free, we have
$$\pi_{*}(x)\cdot \pi_{*}(y)=2\pi_{*}(x\cdot y),$$
for all $x,y\in H^{2}(\mathcal{V},\Z)^{\tau}$.
\end{itemize}
\end{lemme}
\begin{proof}
~~\\
\begin{itemize}
\item[(1) and (2)] are proved in the same way as in Lemma \ref{dot}
\item[(3)] By Proposition \ref{Smithy}, the map $\pi^{*}$ is injective on $H^{*}(\mathcal{V}/\tau,\Z)/(\tors)$.
We assume that $H^{4}(\mathcal{V}/\tau,\Z)$ is torsion free.
Let $x,y$ in $H^{2}(\mathcal{V},\Z)^{\tau}$. We have $\pi^{*}(\pi_{*}(x)\cdot \pi_{*}(y))= 4x\cdot y =\pi^{*}(2\pi_{*}(x\cdot y))$.
So $\pi_{*}(x)\cdot \pi_{*}(y)=2\pi_{*}(x\cdot y)$.
\end{itemize}
\end{proof}

\begin{prop}\label{b2}
We have $b_{2}(M')=16$.
\end{prop}
\begin{proof}
From \eqref{inva}, we see that $\rk H^{2}(S^{[2]},\Z)^{\iota}=15$. Moreover, $s_{1}\circ\iota_{1}=\iota\circ s_{1}$ and $\Sigma_{1}$ is invariant under $\iota_{1}$. Then $H^{2}(N_{1},\Z)^{\iota_{1}}\simeq s_{1}^{*}(H^{2}(S^{[2]},\Z)^{\iota}) \oplus \Z \Sigma_{1}$.
So $\rk H^{2}(N_{1},\Z)^{\iota_{1}}=16$.
By Proposition \ref{Smithy}, $H^{2}(M',\mathbb{Q})\simeq H^{2}(N_{1},\mathbb{Q})^{\iota_{1}}$ and we get $b_{2}(M')=16$.
\end{proof}
We are going to give three other propositions that explain the link between $(H^{2}(S^{[2]},\Z)^{\iota},B_{S^{[2]}})$ and $(H^{2}(M',\Z),B_{M'})$.
\begin{prop}\label{passage}
We have the formula $$B_{M'}(\pi_{1*}(s_{1}^{*}(\alpha),\pi_{1*}(s_{1}^{*}(\beta)))=\sqrt{\frac{24}{C_{M'}}}B_{S^{[2]}}(\alpha,\beta),$$
where $C_{M'}$ is the Fujiki constant of $M'$ and $\alpha$, $\beta$ are in $H^{2}(S^{[2]},\Z)^{\iota}$.
\end{prop}
\begin{proof}
By (1) of Theorem \ref{Form}, we have $$(\pi_{1*}(s_{1}^{*}(\alpha)))^{4}=C_{M'}B_{M'}(\pi_{1*}(s_{1}^{*}(\alpha),\pi_{1*}(s_{1}^{*}(\alpha)))^{2}.$$
And $$\alpha^{4}=3B_{S^{[2]}}(\alpha,\alpha)^{2}.$$
Moreover, by Lemma \ref{inter}, $$(\pi_{1*}(s^{*}(\alpha)))^{4}=8s^{*}(\alpha)^{4}=8\alpha^{4}.$$
By Point (2) of Theorem \ref{Form}, we get the result.
\end{proof}
\begin{prop}\label{delta}
We have $$B_{M'}(\overline{\Sigma}',\overline{\Sigma}')=B_{M'}(\pi_{1*}(s_{1}^{*}(\delta)),\pi_{1*}(s_{1}^{*}(\delta)))=-2\sqrt{\frac{24}{C_{M'}}}.$$
\end{prop}
\begin{proof}
First, by \cite{Beauville} we have $B_{S^{[2]}}(\delta,\delta)=-2$, hence $$B_{M'}(\pi_{1*}(s_{1}^{*}(\delta)),\pi_{1*}(s_{1}^{*}(\delta)))=-2\sqrt{\frac{24}{C_{M'}}}$$ by the last proposition.
Now we need the following proposition which easily follows from Theorem \ref{Form}.
\begin{prop}\label{beauville}
Let $X$ be a projective irreducible symplectic variety of dimension $2n$ with $\codim\sing X\geq 4$.
The equality (1) of Theorem \ref{Form} implies that
$$\alpha_{1}\cdot...\cdot\alpha_{2n}=\frac{c_{X}}{(2n)!}\sum_{\sigma\in S_{2n}}B_{X}(\alpha_{\sigma(1)},\alpha_{\sigma(2)})...B_{X}(\alpha_{\sigma(2n-1)},\alpha_{\sigma(2n)}).$$
for all $\alpha_{i}\in H^{2}(X,\Z)$.
\end{prop}
We need also the following lemma:
\begin{lemme}\label{Fulton}
We have $$\Sigma_{1}^{2}=-s_{1}^{*}(\Sigma).$$
\end{lemme}
\begin{proof}
Consider the following diagram:
$$\xymatrix{
 \Sigma_{1}\ar[d]^{g}\ar@{^{(}->}[r]^{l_{1}}& N_{1} \ar[d]^{s_{1}}\\
   \Sigma \ar@{^{(}->}[r]^{l_{0}}  & S^{[2]},
   }$$
where $l_{0}$ and $l_{1}$ are the inclusions and $g:=s_{1|\Sigma_{1}}$.
By Proposition 6.7 of \cite{Fulton}, we have:
$$s_{1}^{*}l_{0*}(\Sigma)=l_{1*}(c_{1}(E)),$$
where $E:=g^{*}(\mathscr{N}_{\Sigma/S^{[2]}})/\mathscr{N}_{\Sigma_{1}/N_{1}}$.
Hence $$s_{1}^{*}l_{0*}(\Sigma)=c_{1}(g^{*}(\mathscr{N}_{\Sigma/S^{[2]}}))-\Sigma_{1}^{2}.$$
It remains to calculate $c_{1}(g^{*}(\mathscr{N}_{\Sigma/S^{[2]}}))$.
We consider the diagram
$$\xymatrix{
 \Sigma_{0}\ar@{^{(}->}[r]^{l_{0}'}& S\times S \\
  \widetilde{\Sigma_{0}}\ar[u]^{r_{0}} \ar[d]_{p_{0}}\ar@{^{(}->}[r]^{\widetilde{l_{0}}}  &\ar[u]^{r} \ar[d]^{p}\mathscr{S}\\
   \Sigma \ar@{^{(}->}[r]^{l_{0}} & S^{[2]},
   }$$
where $\Sigma_{0}=\left\{\left.(x,i(x))\right|x\in S\right\}$, $r: \mathscr{S}\rightarrow S\times S$ is the blowup in the diagonal of $S\times S$: $\Delta_{0}$, $r_{0}: \widetilde{\Sigma_{0}}\rightarrow\Sigma_{0}$ is the blowup in $\Delta_{0}\cap \Sigma_{0}=8pt$, and $p$, $p_{0}$ are the quotient maps.
Since $\Delta_{0}$ and $\Sigma_{0}$ intersect properly in $S\times S$, $\widetilde{\Sigma_{0}}$ is equal to the total transform of $\Sigma_{0}$ by $r$ in $\mathscr{S}$. Hence $$c_{1}(\mathscr{N}_{\widetilde{\Sigma_{0}}/\mathscr{S}})=c_{1}(r_{0}^{*}(\mathscr{N}_{\Sigma_{0}/S\times S})).$$
But since $\Sigma_{0}\simeq S$, we have $\mathscr{N}_{\Sigma_{0}/S\times S}\simeq \mathscr{T}_{S}$. Hence $c_{1}(\mathscr{N}_{\widetilde{\Sigma_{0}}/\mathscr{S}})=0$.
Since $\widetilde{\Sigma_{0}}$ is also the total transform of $\Sigma$ by $p$ in $\mathscr{S}$, we have
$$c_{1}(\mathscr{N}_{\widetilde{\Sigma_{0}}/\mathscr{S}})=c_{1}(p^{*}(\mathscr{N}_{\Sigma/S^{[2]}})).$$
Hence $c_{1}(\mathscr{N}_{\Sigma/S^{[2]}})=0$.
So
$$\Sigma_{1}^{2}=-s_{1}^{*}l_{0*}(\Sigma).$$
\end{proof}
Now, we can calculate $B_{M'}(\overline{\Sigma}',\overline{\Sigma}')$ from the cup product. 
\begin{align*}
\overline{\Sigma}'^{2}\cdot\pi_{1*}(s_{1}^{*}(\delta))^{2}
&=\frac{C_{M'}}{3}B_{M'}(\overline{\Sigma}',\overline{\Sigma}')\times B_{M'}(\pi_{1*}(s_{1}^{*}(\delta)),\pi_{1*}(s_{1}^{*}(\delta)))\\
&=\frac{C_{M'}}{3}B_{M'}(\overline{\Sigma}',\overline{\Sigma}')\times\left(-2\sqrt{\frac{24}{C_{M'}}}\right)\\
&=-4\sqrt{\frac{2C_{M'}}{3}}B_{M'}(\overline{\Sigma}',\overline{\Sigma}')
\end{align*}
By Lemma \ref{inter}, we have $\overline{\Sigma}'^{2}\cdot\pi_{1*}(s_{1}^{*}(\delta))^{2}=8\Sigma_{1}^{2}\cdot (s_{1}^{*}(\delta))^{2}$.
By the projection formula, $\Sigma_{1}^{2}\cdot (s_{1}^{*}(\delta))^{2}=s_{1*}(\Sigma_{1}^{2})\cdot\delta^{2}$.
Moreover by the last lemma, $s_{1*}(\Sigma_{1}^{2})=-\Sigma$. 
Hence, $$-8\Sigma\cdot\delta^{2}=-4\sqrt{\frac{2C_{M'}}{3}}B_{M'}(\overline{\Sigma}',\overline{\Sigma}').$$
It is possible to understand geometrically the intersection $\Sigma\cdot\Delta^{2}$, where $\Delta=2\delta$ is the diagonal in $S^{[2]}$.
Since $\iota=i^{[2]}$ with $i$ a symplectic involution on $S$, we have $\Sigma=\left\{\left.\xi\in S^{[2]} \right|\Supp \xi=x+i(x),\ x\in S\right\}$
and $\Delta\rightarrow\Delta_{0}$ is a $\mathbb{P}^{1}$-bundle over the diagonal $\Delta_{0}$ in $S^{(2)}$.
We recall that $i$ has 8 fixed points on $S$: $x_{1},...,x_{8}$.
Then we see that $\Delta\cdot\Sigma=\cup_{j=1}^{8}\left\{\left.\xi\in S^{[2]}\right|\Supp \xi =\left\{x_{j}\right\}\right\}$, the union of 8 lines.
Therefore $\Delta^{2}\cdot\Sigma$ is the self-intersection of 8 lines in the K3 surface $\Sigma$ (in fact $\Sigma\simeq Y$).
So $\Delta^{2}\cdot\Sigma=-2\times8$. We get:$$-8\times\frac{-2\times8}{4}=-4\sqrt{\frac{2C_{M'}}{3}}B_{M'}(\overline{\Sigma}',\overline{\Sigma}'),$$
and so
$$B_{M'}(\overline{\Sigma}',\overline{\Sigma}')=-2\sqrt{\frac{24}{C_{M'}}}.$$
\end{proof}
\begin{prop}\label{ortho}
$$B_{M'}(\pi_{1*}(s_{1}^{*}(\alpha)),\overline{\Sigma}')=0,$$
for all $\alpha\in H^{2}(S^{[2]},\Z)^{\iota}$.
\end{prop}
\begin{proof}
We have $\pi_{1*}(s_{1}^{*}(\alpha))^{3}\cdot\overline{\Sigma}'=8s_{1}^{*}(\alpha)^{3}\cdot\Sigma_{1}$ by Lemma \ref{inter},
and $s_{1*}(s_{1}^{*}(\alpha^{3})\cdot\Sigma_{1})=\alpha^{3}\cdot s_{1*}(\Sigma_{1})=0$ by the projection formula.
We conclude by Proposition \ref{beauville}.
\end{proof}

We see from these propositions that the difficulty is the same as in Section 1: we need to understand the image of $s_{1}^{*}(H^{2}(S^{[2]},\Z)^{\iota})$ by $\pi_{1*}$. 
The main question is whether an element $\pi_{1*}(x)$ with $x\in s_{1}^{*}(H^{2}(S^{[2]},\Z)^{\iota})$ is divisible by $2$ in $H^{2}(M',\Z)$ or not.
Let $H^{2}(S^{[2]},\Z)^{\iota}\cong U^{3}\oplus(-2)\oplus E_{8}(-2)$ be the isometry from Proposition \ref{Mong}.
As in the first case, the elements $\pi_{1*}(s_{1}^{*}(x))$ with $s_{1}^{*}(x)\in s_{1}^{*}(E_{8}(-2))\subset s_{1}^{*}(H^{2}(S^{[2]},\Z)^{\iota})$ are divisible by $2$ in $H^{2}(M',\Z)$.
Indeed, we can write $s_{1}^{*}(x)=s_{1}^{*}(t)+s_{1}^{*}(\iota^{*}(t))$ by (2), then $\pi_{1*}(s_{1}^{*}(x))=2\pi_{1*}(s_{1}^{*}(t))$.
What can we say about the divisibility of the elements of $\pi_{1*}(s_{1}^{*}(U^{3}\oplus(-2)))\subset H^{2}(M',\Z)$?

To avoid confusion, we fix an isometry $H^{2}(X,\Z)^{i}\cong U^{3}\oplus E_{8}(-2)$ from Proposition \ref{Mong}, and we will \textbf{write $U^{3}\oplus(-2)\oplus E_{8}(-2)$ instead of $H^{2}(X,\Z)^{i}$} throughout this section. 
\subsection{Plan of the proof}\label{Plan}
Now we will focus on the study of the group $\pi_{1*}(s_{1}^{*}(U^{3}\oplus(-2)))\subset H^{2}(M',\Z)$.
In this subsection, we will give the plan of this study.
We will use the same ideas that in the first example.
First, we need a manifold which will play the role of $Y$.
Consider the blowup $s_{2}:N_{2}\rightarrow N_{1}$ of $N_{1}$ in the 28 points fixed by $\iota_{1}$ and the blowup $\widetilde{r}:\widetilde{M}\rightarrow M'$ of $M'$ in its 28 singulars points.
Denote $\iota_{2}$ the involution induced by $\iota$ on $N_{2}$ and $\pi_{2}:N_{2}\rightarrow N_{2}/\iota_{2}$ the quotient map.
We have $N_{2}/\iota_{2}\simeq \widetilde{M}$. We collect this notation in commutative diagram
$$\xymatrix{
 \widetilde{M}\ar[r]^{\widetilde{r}} & M' \ar[r]^{r'}& M \\
  N_{2}\ar@(dl,dr)[]_{\iota_{2}} \ar[r]^{s_{2}} \ar[u]^{\pi_{2}}& N_{1}\ar@(dl,dr)[]_{\iota_{1}} \ar[r]^{s_{1}} \ar[u]^{\pi_{1}} & S^{[2]}\ar@(dl,dr)[]_{\iota}\ar[u]^{\pi}
   }$$

The manifold $\widetilde{M}$ will play the role of $Y$. In the first part, the fact that $H^{2}(Y,\Z)$ endowed with the cup product was an unimodular lattice have played the central role. But now, the cup product is not a bilinear form on $H^{2}(\widetilde{M},\Z)$. For this reason, we will study $H^{4}(\widetilde{M},\Z)$ endowed with the cup product. This lattice is unimodular. And finally we will deduce from this study the information on $H^{2}(M',\Z)$ that we need. 
We will study $H^{4}(\widetilde{M},\Z)$ in three steps.
\begin{itemize}
\item{1)}
First, we will describe the lattice $H^{4}(S^{[2]},\Z)$ endowed with the cup product.
\item{2)}
Denote $V=S^{[2]}\setminus\Fix \iota$ and $U=M\setminus\sing M= \pi(V)$. 
We will calculate $H^{4}(U,\Z)$ with the same method as in Section \ref{OtherP}.
\item{3)}
Denote $s=s_{2}\circ s_{1}$. Let $x,y\in H^{2}(S^{[2]},\Z)$. Using the results on $H^{4}(\widetilde{M},\Z)$, we
will be able to understand whether the elements $\pi_{2*}(s^{*}(x\cdot y))$ are divisible by 2 or not.
\end{itemize}
From this study, we will deduce the divisibility properties for all the elements $\pi_{2*}(s^{*}(x))\in H^{2}(\widetilde{M},\Z)$
and then conclude as in Section \ref{proof}.
\subsection{The group $H^{4}(S^{[2]},\Z)$}\label{base}
We take the isometry $H^{2}(S^{[2]},\Z)\cong U^{3}\oplus(-2)\oplus E_{8}(-1)\oplus E_{8}(-1)$ given by Proposition \ref{Mong}. Let 
$(u_{k,l})_{k\in\left\{1,2,3\right\}, l\in\left\{1,2\right\}}$ be a basis of $U^{3}$ and $(e_{k,l})_{k\in\left\{1,...,8\right\},}$ $_{l\in\left\{1,2\right\}}$ a basis of $E_{8}(-1)\oplus E_{8}(-1)$. To simplify, we will also use the notation: $$(\gamma_{k})_{k\in\left\{1,...,22\right\}}=(u_{a,b})_{a\in\left\{1,2,3\right\}, b\in\left\{1,2\right\}}\cup(e_{l,p})_{l\in\left\{1,...,8\right\}, p\in\left\{1,2\right\}}.$$ By the proof of Proposition \ref{Mong}, we can write $\gamma_{k}=j(\alpha_{k})$ for all $k\in \left\{1,...,22\right\}$ where $(\alpha_{k})_{k\in\left\{1,...,22\right\}}$ is the corresponding basis of $H^{2}(S,\Z)$. Also denote  $j(v_{k,l})$ $=u_{k,l}$, for all $k\in\left\{1,2,3\right\}$, $l\in\left\{1,2\right\}$.

For $\alpha\in H^{*}(S,\Z)$ and $l\in\Z$, we denote by $\mathfrak{q}_{l}(\alpha)\in\End(H^{*}(S^{[2]},\Z))$ the Nakajima operators \cite{Nakajima} and by $\left|0\right\rangle\in H^{*}(S^{[0]},\Z))$ the unit. We denote also $1\in H^{0}(S,\Z)$ the unit and by $x\in H^{4}(S,\Z)$ the class of a point. 
By Verbitsky \cite{Verbitsky} we know that the cup product map $\sym^{2} H^{2}(S^{[2]},\mathbb{Q})\rightarrow H^{4}(S^{[2]},\mathbb{Q})$ is an isomorphism. Moreover, we have the following theorem by Qin-Wang (\cite{Wang} Theorem 5.4 and Remark 5.6).
\begin{thm}\label{baseS}
The following elements form an integral basis for $H^{4}(S^{[2]},\Z)$:
$$\mathfrak{q}_{1}(1)\mathfrak{q}_{1}(x)\left|0\right\rangle,\ \ \mathfrak{q}_{2}(\alpha_{k})\left|0\right\rangle, \ \ \mathfrak{q}_{1}(\alpha_{k})\mathfrak{q}_{1}(\alpha_{m})\left|0\right\rangle,$$
$$\mathfrak{m}_{1,1}(\alpha_{k})\left|0\right\rangle=\frac{1}{2}(\mathfrak{q}_{1}(\alpha_{k})^{2}-\mathfrak{q}_{2}(\alpha_{k}))\left|0\right\rangle,$$
with $k<m$.
\end{thm}
To get a better idea of this theorem, we will give the following proposition which is Remark 6.7 of \cite{SmithTh}.
\begin{prop}\label{base2}
\begin{itemize}
\item
For all $k\in\left\{1,...,22\right\}$, $$\mathfrak{q}_{2}(\alpha_{k})\left|0\right\rangle=\delta\cdot\gamma_{k},$$
\item
for all $1\leq k\leq m\leq22$, $$\gamma_{k}\cdot\gamma_{m}=(\alpha_{k}\cdot\alpha_{m})\mathfrak{q}_{1}(1)\mathfrak{q}_{1}(x)\left|0\right\rangle+\mathfrak{q}_{1}(\alpha_{k})\mathfrak{q}_{1}(\alpha_{m})\left|0\right\rangle,$$
\item for all $k\in\left\{1,...,22\right\}$, $$\mathfrak{m}_{1,1}(\alpha_{k})\left|0\right\rangle=\frac{\gamma_{k}^{2}-\delta\cdot\gamma_{k}}{2}-\frac{\alpha_{k}^{2}}{2}\mathfrak{q}_{1}(1)\mathfrak{q}_{1}(x)\left|0\right\rangle,$$
\item
denote by $d:S\rightarrow S^{2}$ the diagonal embedding. We denote the push-forward map followed by the Künneth isomorphism by $d_{*}:H^{*}(S,\Z)\rightarrow H^{*}(S,\Z)\otimes H^{*}(S,\Z)$. We write $d_{*}(1)= \sum_{k,m} \mu_{k,m} \alpha_{k}\otimes\alpha_{m} +x\otimes 1 +1\otimes x$, $\mu_{k,m}\in\Z$. Since $\mu_{k,m}=\mu_{m,k}$, one has:
$$\delta^{2}=\sum_{i<j}{\mu_{i,j}\mathfrak{q}_{1}(\alpha_{i})\mathfrak{q}_{1}(\alpha_{j})\left|0\right\rangle}+\frac{1}{2}\sum_{i}{\mu_{i,i}\mathfrak{q}_{1}(\alpha_{i})^{2}\left|0\right\rangle}+\mathfrak{q}_{1}(1)\mathfrak{q}_{1}(x)\left|0\right\rangle.$$
\end{itemize}
\end{prop}
\begin{proof}
We have $$\frac{1}{2}\mathfrak{q}_{2}(1)\left|0\right\rangle=\delta,\ \ \ \ \mathfrak{q}_{1}(1)\mathfrak{q}_{1}(\alpha_{k})\left|0\right\rangle=j(\alpha_{k})=\gamma_{k},$$
for all $k\in\left\{1,...,22\right\}$.
The cup product map $\sym^{2} H^{2}(S^{[2]},\mathbb{Q})\rightarrow H^{4}(S^{[2]},\mathbb{Q})$ can be computed explicitly by using the algebraic model constructed by Lehn-Sorger \cite{Lehn}:
\begin{itemize}
\item[1)] for $\alpha\in H^{2}(S,\Z)$, $\frac{1}{2}\mathfrak{q}_{2}(1)\left|0\right\rangle\cdot\mathfrak{q}_{1}(1)\mathfrak{q}_{1}(\alpha)\left|0\right\rangle=\mathfrak{q}_{2}(\alpha)\left|0\right\rangle,$
\item[2)] for $\alpha,\beta\in H^{2}(S,\Z)$,
$$\mathfrak{q}_{1}(1)\mathfrak{q}_{1}(\alpha)\left|0\right\rangle\cdot\mathfrak{q}_{1}(1)\mathfrak{q}_{1}(\beta)\left|0\right\rangle=(\alpha\cdot\beta)\mathfrak{q}_{1}(1)\mathfrak{q}_{1}(x)\left|0\right\rangle+\mathfrak{q}_{1}(\alpha)\mathfrak{q}_{1}(\beta)\left|0\right\rangle,$$
\end{itemize}
This implies the Proposition.
\end{proof}
We will need also a proposition on the cup product of the previous elements.
\begin{prop}\label{sigma}
We have 
\begin{itemize}
\item[1)] $$\mathfrak{q}_{1}(1)\mathfrak{q}_{1}(x)\left|0\right\rangle\cdot\mathfrak{q}_{2}(\alpha_{k})\left|0\right\rangle=\mathfrak{q}_{1}(1)\mathfrak{q}_{1}(x)\left|0\right\rangle\cdot\mathfrak{q}_{1}(\alpha_{k})\mathfrak{q}_{1}(\alpha_{l})\left|0\right\rangle=0$$
for all $(k,l)\in \left\{1,...,22\right\}^{2}$, and
$$\mathfrak{q}_{1}(1)\mathfrak{q}_{1}(x)\left|0\right\rangle\cdot\mathfrak{q}_{1}(1)\mathfrak{q}_{1}(x)\left|0\right\rangle=1.$$
\item[2)]
$$\Sigma\cdot\mathfrak{q}_{1}(1)\mathfrak{q}_{1}(x)\left|0\right\rangle=1, \ \ \ \Sigma\cdot\mathfrak{q}_{2}(\alpha_{k})\left|0\right\rangle=0,$$ and 
$$\Sigma\cdot\mathfrak{q}_{1}(\alpha_{k})\mathfrak{q}_{1}(\alpha_{l})\left|0\right\rangle=\alpha_{k}\cdot i^{*}\alpha_{l},$$
for all $(k,l)\in \left\{1,...,22\right\}^{2}$.
\end{itemize}
\end{prop}
\begin{proof}
We recall the definition of Nakajima's operators.
Let 
$$Q^{[m+n,n]}=\left\{\left.(\xi,x,\eta)\in S^{[m+n]}\times S\times S^{[m]}\right|\xi\supset \eta,\ \Supp(I_{\eta}/I_{\xi})=\left\{x\right\}\right\}$$
We have 
$$\mathfrak{q}_{n}(\alpha)(A)=\widetilde{p}_{1*}\left(\left[Q^{[m+n,m]}\right]\cdot\widetilde{\rho}^{*}\alpha\cdot \widetilde{p}_{2}^{*}A\right),$$
for $A\in H^{*}(S^{[m]})$ and $\alpha\in H^{*}(S)$ , where $\widetilde{p}_{1}$, $\widetilde{\rho}$, $\widetilde{p}_{2}$ are the projections from $S^{[m+n]}\times S\times S^{[m]}$ to $S^{[m+n]}$, $S$, $S^{[m]}$ respectively.
Then we find that $\mathfrak{q}_{1}(1)\mathfrak{q}_{1}(x)\left|0\right\rangle$
corresponds to the cycle $\left\{\left.\xi\in S^{[2]}\right|\Supp \xi\ni x\right\}$.
The element $\mathfrak{q}_{1}(\alpha_{k})\mathfrak{q}_{1}(\alpha_{m})\left|0\right\rangle$
corresponds to the cycle $\left\{\left.\xi\in S^{[2]}\right|\Supp \xi=x+y,\ x\in \alpha_{k},\ y\in\alpha_{m}\right\}$.
And $\mathfrak{q}_{2}$ $(\alpha_{k})\left|0\right\rangle$ corresponds to the cycle $\left\{\left.\xi\in S^{[2]}\right|\Supp \xi=\left\{x\right\},\ x\in \alpha_{k}\right\}$.
This implies 1).
Since $\Sigma=\left\{\left.\xi\in S^{[2]}\right|\Supp \xi=x+i(x)\right\}$, we deduce 2).
\end{proof}
\subsection{The group $H^{4}(U,\Z)$}
We will follow the same method as in Section \ref{OtherP} to prove the following proposition.
\begin{prop}\label{U}
We have an isomorphism of group:
$$H^{4}(U,\Z)\simeq H^{4}(S^{[2]},\Z)^{\iota}/\Z\Sigma\oplus(\Z/2\Z)^{8}.$$
\end{prop}
We begin by determining $H^{3}(V,\Z)$ and $H^{4}(V,\Z)$.
\subsubsection{The groups $H^{3}(V,\Z)$ and $H^{4}(V,\Z)$}
\begin{lemme}
We have:
\begin{itemize}
\item[(1)] $H^{2}(V,\Z)=H^{2}(S^{[2]},\Z)$,
\item[(2)] $H^{3}(V,\Z)=0$,
\item[(3)] $H^{4}(V,\Z)=H^{4}(S^{[2]},\Z)/\Z \Sigma.$
\end{itemize}
\end{lemme}
\begin{proof}
\begin{itemize}
\item[(1)] 
Consider the following exact sequence:
$$\xymatrix@C=10pt{ H^{2}(S^{[2]},V,\Z)\ar[r] & H^{2}(S^{[2]},\Z)\ar[r] & H^{2}(V,\Z)\ar[r] & H^{3}(S^{[2]},V,\Z)}.$$
By Thom's isomorphism, we have $H^{2}(S^{[2]},V,\Z)=H^{3}(S^{[2]},V,\Z)$ $=0$.
\item[(2) and (3)]
Consider the following exact sequence:
$$\xymatrix@C=10pt@R=0pt{H^{3}(S^{[2]},\Z)\ar[r]& H^{3}(V,\Z)\ar[r] & H^{4}(S^{[2]},V,\Z)\ar[r] & H^{4}(S^{[2]},\Z)\ar[r] & H^{4}(V,\Z)\\
 & & & \ar[r]& H^{5}(S^{[2]},V,\Z).}$$
We have $H^{3}(S^{[2]},\Z)=0$, and by Thom's isomorphism $H^{4}(S^{[2]},V,$ $\Z)=H^{0}(\Sigma,\Z)=\Z$, $H^{5}(S^{[2]},V,\Z)=H^{1}(\Sigma,\Z)=0$.
We can see that $H^{3}(V,\Z)=0$. For this, consider the exact sequence of homology:
$$\xymatrix@C=15pt{0& H_{3}(V)\ar[l] & H_{4}(S^{[2]},V)\ar[l]^{\partial} & H_{4}(S^{[2]})\ar[l]^{p} & H_{4}(V)\ar[l]^{l} & \ar[l]0}.$$

We know that $\rk H_{4}(S^{[2]},V)=1$. Let $\left[\Sigma\right]\in H_{4}(S^{[2]})$ be the cycle corresponding to $\Sigma$. We have $p(\left[\Sigma\right])\neq 0$ and $\partial(p(\left[\Sigma\right]))=0$, hence $H_{3}(V,\Z)$ is torsion.
Since $H^{3}(V,\Z)$ is torsion free, by the universal coefficient theorem, we have $H^{3}(V,\Z)=0$.
Hence $$H^{4}(V,\Z)=H^{4}(S^{[2]},\Z)/\Z \Sigma.$$
\end{itemize}
\end{proof}
\subsubsection{Calculation with equivariant cohomology}
Now we will calculate $H^{4}(U,\Z)$ by using equivariant cohomology.
Let $G:=\left\{\id, \iota\right\}$. We have $$H^{4}(U,\Z)\simeq H_{G}^{4}(V,\Z)$$
and
$$E_{2}^{p,q}:=H^{p}(G;H^{q}(V,\Z))\Rightarrow H_{G}^{p+q}(V,\Z).$$
By using the resolution by $\Z[G]$-modules
$$\xymatrix@C=10pt{...\ar[r]^{\iota-1}& \Z[G]\ar[r]^{\iota+1} & \Z[G]\ar[r]^{\iota-1} & \Z[G]\ar[r] & \Z},$$
we get $$H^{0}(G;H^{q}(V,\Z))=H^{q}(V,\Z)^{\iota},$$
\begin{align*}
&H^{2k+1}(G;H^{q}(V,\Z))\\
&=\left\{\left. m\in H^{q}(V,\Z)\right|\iota^{*}(m)=-m\right\}/\left\{\left. \iota^{*}(m)-m\right|m\in H^{q}(V,\Z) \right\},\\
&H^{2k+2}(G;H^{q}(V,\Z))\\
&=\left\{\left. m\in H^{q}(V,\Z)\right|\iota^{*}(m)=m\right\}/\left\{\left. \iota^{*}(m)+m\right|m\in H^{q}(V,\Z) \right\},
\end{align*}
for all $k\in\mathbb{N}$.
The lemma follows.
\begin{lemme}\label{calcul}
We have $$H^{2}(G;H^{2}(V,\Z))=(\Z/2\Z)^{7},\ \ \ \ H^{3}(G;H^{2}(V,\Z))=0,\ \ \ \ H^{1}(G;H^{4}(V,\Z))=0.$$
\end{lemme}
\begin{proof}
The equalities $H^{2}(G;H^{2}(V,\Z))=(\Z/2\Z)^{7}$ and $H^{3}(G;H^{2}(V,\Z))=0$ follow from Proposition \ref{Mong}.
It remains to show that $H^{1}(G;H^{4}(V,\Z))=0$.
First note that:
$$\iota^{*}(\mathfrak{q}_{2}(\alpha_{k})\left|0\right\rangle)=\mathfrak{q}_{2}(i^{*}\alpha_{k})\left|0\right\rangle, \ \ \iota^{*}(\mathfrak{q}_{1}(\alpha_{k})\mathfrak{q}_{1}(\alpha_{j})\left|0\right\rangle)=\mathfrak{q}_{1}(i^{*}\alpha_{k})\mathfrak{q}_{1}(i^{*}\alpha_{j})\left|0\right\rangle,$$ 
$$\iota^{*}(\mathfrak{m}_{1,1}(\alpha_{k})\left|0\right\rangle)=\mathfrak{m}_{1,1}(i^{*}\alpha_{k})\left|0\right\rangle,\ \ \ \ \iota^{*}(\mathfrak{q}_{1}(1)\mathfrak{q}_{1}(x)\left|0\right\rangle)=\mathfrak{q}_{1}(1)\mathfrak{q}_{1}(x)\left|0\right\rangle,$$
Let $x\in H^{4}(S^{[2]},\Z)$ such that $\iota^{*}(x)=-x$. By Theorem \ref{baseS}, we can write
\begin{align*}
x=&\sum_{0\leq k< j\leq 22}{\lambda_{k,j}\mathfrak{q}_{1}(\alpha_{k})\mathfrak{q}_{1}(\alpha_{j})\left|0\right\rangle}\\
&+\sum_{0\leq k\leq 22}{\eta_{k}\mathfrak{q}_{2}(\alpha_{k})\left|0\right\rangle}\\
&+\sum_{0\leq k\leq 22}{\nu_{k}\mathfrak{m}_{1,1}(\alpha_{k})\left|0\right\rangle}\\
&+y\mathfrak{q}_{1}(1)\mathfrak{q}_{1}(x)\left|0\right\rangle,
\end{align*}
where the $\lambda_{k,j}$, $\eta_{k}$, $\nu_{k}$  are in $\Z$.
By Proposition \ref{Mong}, 
\begin{align*}
&\iota^{*}(\sum_{0\leq k< j\leq 22}{\lambda_{k,j}\mathfrak{q}_{1}(\alpha_{k})\mathfrak{q}_{1}(\alpha_{j})\left|0\right\rangle})\\
&=\sum_{0\leq k< j\leq 6}{\lambda_{k,j}\mathfrak{q}_{1}(\alpha_{k})\mathfrak{q}_{1}(\alpha_{j})\left|0\right\rangle}+\sum_{0\leq k\leq 6<j\leq 14}{\lambda_{k,j}\mathfrak{q}_{1}(\alpha_{k})\mathfrak{q}_{1}(\alpha_{8+j})\left|0\right\rangle}\\
&+\sum_{0\leq k\leq 6,15\leq j\leq 22}{\lambda_{k,j}\mathfrak{q}_{1}(\alpha_{k})\mathfrak{q}_{1}(\alpha_{j-8})\left|0\right\rangle}+\sum_{7\leq k< j\leq 14}{\lambda_{k,j}\mathfrak{q}_{1}(\alpha_{k+8})\mathfrak{q}_{1}(\alpha_{j+8})\left|0\right\rangle}\\
&+\sum_{7\leq k\leq 14< j\leq 22}{\lambda_{k,j}\mathfrak{q}_{1}(\alpha_{k+8})\mathfrak{q}_{1}(\alpha_{j-8})\left|0\right\rangle}+\sum_{15\leq k< j\leq 22}{\lambda_{k,j}\mathfrak{q}_{1}(\alpha_{k-8})\mathfrak{q}_{1}(\alpha_{j-8})\left|0\right\rangle}.\\
\end{align*}
Since $x$ is anti-invariant and by Theorem \ref{baseS},
$$\iota^{*}\left(\sum_{0\leq k< j\leq 22}{\lambda_{k,j}\mathfrak{q}_{1}(\alpha_{k})\mathfrak{q}_{1}(\alpha_{j})\left|0\right\rangle}\right)=-\sum_{0\leq k< j\leq 22}{\lambda_{k,j}\mathfrak{q}_{1}(\alpha_{k})\mathfrak{q}_{1}(\alpha_{j})\left|0\right\rangle}.$$
Hence
\begin{align*}
&\sum_{0\leq k< j\leq 6}{\lambda_{k,j}\mathfrak{q}_{1}(\alpha_{k})\mathfrak{q}_{1}(\alpha_{j})\left|0\right\rangle}+\sum_{0\leq k\leq 6<j\leq 14}{\lambda_{k,j}\mathfrak{q}_{1}(\alpha_{k})\mathfrak{q}_{1}(\alpha_{8+j})\left|0\right\rangle}\\
&+\sum_{0\leq k\leq 6,15\leq j\leq 22}{\lambda_{k,j}\mathfrak{q}_{1}(\alpha_{k})\mathfrak{q}_{1}(\alpha_{j-8})\left|0\right\rangle}+\sum_{7\leq k< j\leq 14}{\lambda_{k,j}\mathfrak{q}_{1}(\alpha_{k+8})\mathfrak{q}_{1}(\alpha_{j+8})\left|0\right\rangle}\\
&+\sum_{7\leq k\leq 14< j\leq 22}{\lambda_{k,j}\mathfrak{q}_{1}(\alpha_{k+8})\mathfrak{q}_{1}(\alpha_{j-8})\left|0\right\rangle}+\sum_{15\leq k< j\leq 22}{\lambda_{k,j}\mathfrak{q}_{1}(\alpha_{k-8})\mathfrak{q}_{1}(\alpha_{j-8})\left|0\right\rangle}\\
&=-\sum_{0\leq k< j\leq 22}{\lambda_{k,j}\mathfrak{q}_{1}(\alpha_{k})\mathfrak{q}_{1}(\alpha_{j})\left|0\right\rangle}.
\end{align*}
Hence
\begin{itemize}
\item
$\lambda_{k,j}=0$ for $0\leq k \leq j\leq 6$.
\item
$\lambda_{k,j}=-\lambda_{k,j+8}$ for $0\leq k\leq 6<j\leq 14$.
\item
$\lambda_{k,j}=-\lambda_{k+8,j+8}$ for $7\leq k\leq j\leq 14$.
\item
$\lambda_{k,j}=-\lambda_{j-8,k+8}$ for $7\leq k\leq 14< j\leq 22$.
\end{itemize}
Therefore after the same calculation on $\eta_{k}$ and $\nu_{k}$, we get:
\begin{align*}
x&=\sum_{1\leq k\leq 6<j\leq 14}{\lambda_{k,j}(\mathfrak{q}_{1}(\alpha_{k})\mathfrak{q}_{1}(\alpha_{j})\left|0\right\rangle-\mathfrak{q}_{1}(\alpha_{k})\mathfrak{q}_{1}(i^{*}\alpha_{j})\left|0\right\rangle)}\\
&+\sum_{7\leq k\leq 14}{\eta_{k}(\mathfrak{q}_{2}(\alpha_{k})\left|0\right\rangle-\mathfrak{q}_{2}(i^{*}\alpha_{k})\left|0\right\rangle)}\\
&+\sum_{7\leq k\leq j\leq14}{\lambda_{k,j}(\mathfrak{q}_{1}(\alpha_{k})\mathfrak{q}_{1}(\alpha_{j})\left|0\right\rangle-\mathfrak{q}_{1}(i^{*}\alpha_{k})\mathfrak{q}_{1}(i^{*}\alpha_{j})\left|0\right\rangle)}\\
&+\sum_{7\leq k\leq 14}{\nu_{k}(\mathfrak{m}_{1,1}(\alpha_{k})\left|0\right\rangle-\mathfrak{m}_{1,1}(i^{*}\alpha_{k})\left|0\right\rangle)}\\
&+\sum_{k=8}^{14}{\sum_{j=15}^{7+k}{\lambda_{k,j}(\mathfrak{q}_{1}(\alpha_{k})\mathfrak{q}_{1}(\alpha_{j})\left|0\right\rangle-\mathfrak{q}_{1}(i^{*}\alpha_{k})\mathfrak{q}_{1}(i^{*}\alpha_{j})\left|0\right\rangle)}},
\end{align*}
Since $H^{4}(V,\Z)=H^{4}(S^{[2]},\Z)/\Z \Sigma$, we get $H^{1}(G;H^{4}(V,\Z))=0$.
\end{proof}
The only remaining difficulty for our calculation is now to see that the differential of the page $E_{3}$,
$D_{3}:H^{2}(G;H^{2}(V,\Z))\rightarrow H^{0}(G;H^{5}(V,\Z))$, and that of the page $E_{5}$,
$D_{5}:H^{4}(G;H^{0}(V,\Z))\rightarrow H^{0}(G;H^{5}(V,\Z))$, are trivial.
This is true because $H^{5}(V,\Z)$ is torsion free and $H^{2}(G;H^{2}(V,\Z))=(\Z/2\Z)^{7}$, $H^{4}(G;H^{0}(V,\Z))=\Z/2\Z$.

Hence the spectral sequence gives 
\begin{align*}
H_{G}^{4}(V,\Z)&=H^{0}(G;H^{4}(V,\Z))\oplus H^{1}(G;H^{3}(V,\Z))\oplus H^{2}(G;H^{2}(V,\Z))\\
&\oplus H^{3}(G;H^{1}(V,\Z))\oplus H^{4}(G;H^{0}(V,\Z)).
\end{align*}
It follows that
$$H_{G}^{4}(V,\Z)=H^{4}(V,\Z)^{\iota}\oplus(\Z/2\Z)^{8}.$$
And we get Proposition \ref{U}:
$$H^{4}(U,\Z)\simeq H^{4}(S^{[2]},\Z)^{\iota}/\Z\Sigma\oplus(\Z/2\Z)^{8}.$$

Now, we will determine the torsion of $H^{4}(U,\Z)$.
\begin{prop}\label{torsion}
The torsion group of $H^{4}(U,\Z)$ is $(\Z/2\Z)^{8}$.
\end{prop}
\begin{proof}
By (2) of Proposition \ref{sigma}, $\Sigma$ is not divisible by 2, hence $H^{4}(S^{[2]},\Z)^{\iota}/\Sigma$ is torsion free.
\end{proof}
\subsubsection{The groups $H^{2}(U,\Z)$ and $H^{3}(U,\Z)$}
\begin{prop}\label{H3}
We have $H^{3}(U,\Z)=0$.
\end{prop}
\begin{proof}
Indeed, all the groups $H^{3}(G;H^{0}(V,\Z))$, $H^{2}(G;H^{1}(V,\Z))$, $H^{1}(G;H^{2}(V,$ $\Z))$ and $H^{0}(G;H^{3}(V,\Z))$ are trivial (the group $H^{1}(G;H^{2}(V,\Z))$ is trivial by Proposition \ref{Mong}).
\end{proof}
\begin{prop}\label{H2}
We have an isomorphism of group:
$$H^{2}(U,\Z)\simeq H^{2}(S^{[2]},\Z)^{\iota}\oplus\Z/2\Z.$$
\end{prop}
\begin{proof}
We have $H^{2}(V,\Z)=H^{2}(S^{[2]},\Z)$
and by Proposition \ref{Mong}, $H^{1}(G,H^{2}(S^{[2]}$ $,\Z))=0$.
Hence the previous spectral sequence gives 
$$H_{G}^{2}(V,\Z)=H^{0}(G;H^{2}(V,\Z))\oplus H^{1}(G;H^{1}(V,\Z))\oplus H^{2}(G;H^{0}(V,\Z)).$$
It follows that
$$H_{G}^{2}(V,\Z)=H^{2}(V,\Z)^{i}\oplus \Z/2\Z.$$
Hence
$$H^{2}(U,\Z)\simeq H^{2}(S^{[2]},\Z)^{i}\oplus \Z/2\Z.$$
\end{proof}
\subsection{The group $H^{4}(\widetilde{M},\Z)$}
\subsubsection{The group $H^{4}(N_{2},\Z)$}
We begin with some notation. Denote $s=s_{2}\circ s_{1}$ (see the diagram of Section \ref{Plan} and Section \ref{Focus}). 
Also denote $\Sigma_{2}=s_{2}^{*}(\Sigma_{1})=s^{-1}(\Sigma)$, $(p_{k})_{1\leq k\leq 28}$ the 28 points fixed by $\iota$, $E_{k}=s^{-1}(p_{k})$ for all $1\leq k\leq 28$ and $h_{k}=c_{1}(\mathcal{O}_{E_{k}}(1))$. Denote $l:\Sigma_{2}\hookrightarrow N_{2}$ and $f_{k}:E_{k}\hookrightarrow N_{2}$ the injections.
By Theorem 7.31 of \cite{Voisin}, we have the following proposition.
\begin{prop}\label{vois}
We have the following two isomorphisms
$$\xymatrix{
H^{2}(S^{[2]},\Z)\oplus H^{0}(\Sigma,\Z)\oplus(\oplus_{k} H^{0}(p_{k},\Z))\ar[rrrr]^{\ \ \ \ \ \ \ \ \ \ \ s^{*}+l_{*}\circ s_{|\Sigma_{2}}^{*}+f_{k*}\circ s_{|E_{k}}^{*}} & & & & H^{2}(N_{2},\Z),}$$
$$\xymatrix{
H^{4}(S^{[2]},\Z)\oplus H^{2}(\Sigma,\Z)\oplus(\oplus_{k} H^{0}(p_{k},\Z))\ar[rrrr]^{\ \ \ \ \ \ \ \ \ \ \ \ s^{*}+l_{*}\circ s_{|\Sigma_{2}}^{*}+f_{k*}\circ h_{k}\circ s_{|E_{k}}^{*}} & & & & H^{4}(N_{2},\Z),}$$
where $h_{k}$ is the morphism given by the cup product by $h_{k}$.
\end{prop}
Let $(a_{k})_{1\leq k \leq 22}$ be an integer basis of $H^{2}(\Sigma,\Z)$. Denote $\theta_{k}=l_{*}\circ s_{|\Sigma_{2}}^{*}(a_{k})$ for all $1\leq k\leq22$.
For a better understanding of $H^{4}(N_{2},\Z)$, we will calculate the cup product on $H^{2}(N_{2},\Z)$.
We give the following proposition.
\begin{prop}\label{intersections}
We have:
\begin{itemize}
\item[1)]
$E_{l}\cdot E_{k}=0$ if $l\neq k$, $E_{l}^{2}= -f_{l*}(h_{l})$, $E_{l}^{4}=-1$ and $E_{l}\cdot z=0$ for all $(l,k)\in \left\{1,...,28\right\}^{2}$
and for all $z\in s^{*}(H^{4}(S^{[2]},\Z))$;
\item[2)]
$\Sigma_{2}\cdot E_{k}=0$ for all $k\in\left\{1,...,28\right\}$ and $\Sigma_{2}^{2}=-s^{*}(\Sigma)$;
\item[3)]
$\theta_{k}\cdot z=0$ for all $k\in \left\{1,...,22\right\}$ and $z\in s^{*}(H^{4}(S^{[2]},\Z))$.
\item[4)]
Denote $\sigma_{x}:=\Sigma_{2}\cdot s^{*}(x)$ for all $x\in H^{2}(S^{[2]},\Z)$. We have
$$\Sigma_{2}^{2}\cdot s^{*}(x)\cdot s^{*}(y)=-2B_{S^{[2]}}(x,y)$$
for all $(x,y)\in H^{2}(S^{[2]},\Z)^{\iota}$.
Hence $$\rk \left\langle\left\{\left. \sigma_{x} \right| x\in H^{2}(S^{[2]},\Z)^{\iota} \right\}\right\rangle=\rk H^{2}(S^{[2]},\Z)^{\iota}=15.$$
\item[5)]
Let $T$ be the sublattice of $(H^{4}(N_{2},\Z),\cdot)$ generated by the set $$\left\{\left.\theta_{i}\right|i\in\left\{1,...,22\right\}\right\}\cup\left\{\left.E_{k}^{2}\right|k\in\left\{1,...,28\right\}\right\}.$$
Then $T$ is unimodular.
\end{itemize}
\end{prop}
\begin{proof}
\begin{itemize}
\item[1)]
Let $k\in \left\{1,...,28\right\}$.
We have $$\omega_{E_{k}}=\omega_{N_{2}}\otimes\mathscr{N}_{E_{k}/N_{2}}.$$
Since $E_{k}\simeq \mathbb{P}^{3}$,
\begin{align*}
\mathcal{O}_{E_{k}}(-4)&=\mathcal{O}_{N_{2}}(3\sum_{j=1}^{28}{E_{j}+\Sigma_{2}})\otimes\mathcal{O}_{N_{2}}(E_{k})\otimes\mathcal{O}_{E_{k}}\\
&=\mathcal{O}_{N_{2}}(4E_{k})\otimes\mathcal{O}_{E_{k}}.
\end{align*}
We get $E_{k}^{2}= -f_{k*}(h_{k})$, where $h_{k}$ is the class of a hyperplane in $E_{k}\simeq \mathbb{P}^{3}$.
Hence $E_{k}^{4}=c_{1}(\mathscr{N}_{E_{k}/N_{2}})^{3}=(-h_{k})^{3}=-1$.
\item[2)]
We have $\Sigma_{2}^{2}=s_{2}^{*}(\Sigma_{1}^{2})$, so the result follows from Lemma \ref{Fulton}.
\item[3)]
If we take $z=s^{*}(y)$, then $s_{*}(z\cdot\theta_{i})=y\cdot s_{*}(\theta_{i})$ by the projection formula.
Since $s_{*}(\theta_{i})=0$, we have $s_{*}(z\cdot\theta_{i})=0$, and then $z\cdot\theta_{i}=0$.
\item[4)]
We have
$$\Sigma_{2}^{2}\cdot s^{*}(x)\cdot s^{*}(y)=\Sigma_{1}^{2}\cdot s_{1}^{*}(x)\cdot s_{1}^{*}(y).$$
By Proposition \ref{inter},
$$\Sigma_{1}^{2}\cdot s_{1}^{*}(x)\cdot s_{1}^{*}(y)=\frac{1}{8}\overline{\Sigma}'^{2}\cdot \pi_{1*}( s_{1}^{*}(x))\cdot \pi_{1*}( s_{1}^{*}(y)).$$
By Proposition \ref{beauville} and \ref{ortho},
$$\Sigma_{1}^{2}\cdot s_{1}^{*}(x)\cdot s_{1}^{*}(y)=\frac{C_{M'}}{24}B_{M'}(\overline{\Sigma}',\overline{\Sigma}')\times B_{M'}(\pi_{1*}(s_{1}^{*}(x)),\pi_{1*}(s_{1}^{*}(y))).$$
Hence, by Proposition \ref{passage} and \ref{delta},
\begin{align*}
\Sigma_{1}^{2}\cdot s_{1}^{*}(x)\cdot s_{1}^{*}(y)&= \frac{C_{M'}}{24}\times\left[-2\sqrt{\frac{24}{C_{M'}}}\right]\times\sqrt{\frac{24}{C_{M'}}}B_{S^{[2]}}(x,y)\\
&=-2B_{S^{[2]}}(x,y).
\end{align*}
\item[5)]
By Proposition \ref{vois}, point 1) and 3), we can write $$H^{4}(N_{2},\Z)=s^{*}(H^{4}(S^{[2]},\Z))\oplus^{\bot} T.$$
Since $N_{2}$ and $S^{[2]}$ are smooth, $(H^{4}(N_{2},\Z),\cdot)$ and $(H^{4}(S^{[2]},\Z),\cdot)$ are unimodular, hence $T$ is unimodular.
\end{itemize}
\end{proof}
From this proposition and (3) of Lemma \ref{inter}, we deduce a similar proposition on the cohomology of $\widetilde{M}$.
We will denote $$h_{k}=-E_{k}^{2},\ \ \ \ \widetilde{\Sigma}:=\pi_{2*}(\Sigma_{2}),\ \ \ \ D_{k}:=\pi_{2*}(E_{k}),$$ $$\widetilde{h_{k}}:=\pi_{2*}(h_{k}),\ \ \ \ \widetilde{\sigma_{x}}:=\pi_{2*}(\sigma_{x}),\ \ \ \ \widetilde{\theta_{l}}:=\pi_{2*}(\theta_{l}),$$ for $k\in\left\{1,...,28\right\}$, $l\in\left\{1,...,22\right\}$ and  $x\in H^{2}(S^{[2]},\Z)$. 
\begin{prop}\label{interr}
We have:
\begin{itemize}
\item{1)}
$D_{l}\cdot D_{k}=0$ if $l\neq k$, $D_{l}^{2}= -2\widetilde{h_{l}}$, and $D_{l}\cdot z=0$ for all $(l,k)\in \left\{1,...,28\right\}^{2}$
and for all $z\in \pi_{2*}(s^{*}(H^{4}(S^{[2]},\Z)^{\iota}))$;
\item{2)}
$\widetilde{\Sigma}\cdot D_{k}=0$ for all $k\in\left\{1,...,28\right\}$ and $\widetilde{\Sigma}^{2}=-2\pi_{2*}(s^{*}(\Sigma))$;
\item{3)}
$\widetilde{\theta}_{k}\cdot z=0$ for all $k\in \left\{1,...,22\right\}$ and $z\in \pi_{2*}(s^{*}(H^{4}(S^{[2]},\Z)^{\iota}))$,
\item{4)}
$\widetilde{\Sigma}\cdot \pi_{2*}(s^{*}(x))=2\widetilde{\sigma_{x}}$ for all $x\in H^{2}(S^{[2]},\Z)^{\iota}$. Moreover 
$$\widetilde{\Sigma}^{2}\cdot \pi_{2*}(s^{*}(x))\cdot \pi_{2*}(s^{*}(y))=-16B_{S^{[2]}}(x,y)$$
for all $(x,y)\in H^{2}(S^{[2]},\Z)^{\iota}$.
Hence
$$\rk \left\langle\left\{\left. \widetilde{\sigma_{x}} \right| x\in H^{2}(S^{[2]},\Z)^{\iota} \right\}\right\rangle=\rk H^{2}(S^{[2]},\Z)^{\iota}=15.$$
\end{itemize}
\end{prop}
Now we will use the same techniques as in Section \ref{end1} and Proof of Lemma \ref{U2}. We will study $\pi_{2*}(s^{*}(H^{4}(S^{[2]},\Z)))\subset H^{4}(\widetilde{M},\Z)$ and $\pi_{2*}(T)\subset H^{4}(\widetilde{M},\Z)$.
Denote $K:=\pi_{2*}(s^{*}(H^{4}(S^{[2]},\Z)))$. As the first step, we will look at the sublattice 
$H^{4}(S^{[2]},\Z)^{\iota}$ in $H^{4}(S^{[2]},\Z)$ (all the groups are endowed with the cup product). We will find its rank, its discriminant and an integral basis. In the second step, we will determine the discriminant of $K$ and will enumerate the elements that we already know to be divisible by 2. In the third step, we will study $\pi_{2*}(T)$. We will find a primitive overlattice of $\pi_{2*}(T)$ and we will calculate its discriminant. And, in the last step, we will be able to give our key lemma \ref{key} on $H^{4}(\widetilde{M},\Z)$.
\subsubsection{The lattice $H^{4}(S^{[2]},\Z)^{\iota}$}
We use the notation from Section \ref{base} and Lemma \ref{calcul}.
\begin{lemme}\label{invariant}
\begin{itemize}
\item[1)]
The following elements form an integral basis of $H^{4}(S^{[2]},\Z)^{\iota}$:
\begin{itemize}
\item[a)]
$$\mathfrak{q}_{2}(\alpha_{k})\left|0\right\rangle, \ \ \mathfrak{q}_{1}(\alpha_{k})\mathfrak{q}_{1}(\alpha_{l})\left|0\right\rangle, \ \ \ \mathfrak{m}_{1,1}(\alpha_{k})\left|0\right\rangle,$$
for $1\leq k<l\leq 6$;
\item[b)]
$$\mathfrak{q}_{1}(\alpha_{k})\mathfrak{q}_{1}(\alpha_{l})\left|0\right\rangle+\mathfrak{q}_{1}(\alpha_{k})\mathfrak{q}_{1}(i^{*}\alpha_{l})\left|0\right\rangle),\ \ \ 
\mathfrak{q}_{2}(\alpha_{l})\left|0\right\rangle+\mathfrak{q}_{2}(i^{*}\alpha_{l})\left|0\right\rangle,$$
for $k\in\left\{1,...,6\right\}$ and $l\in\left\{7,...,14\right\}$;
\item[c)]
$$\mathfrak{q}_{1}(\alpha_{k})\mathfrak{q}_{1}(\alpha_{l})\left|0\right\rangle+\mathfrak{q}_{1}(i^{*}\alpha_{k})\mathfrak{q}_{1}(i^{*}\alpha_{l})\left|0\right\rangle,\ \ \ \ \ \mathfrak{m}_{1,1}(\alpha_{k})\left|0\right\rangle+\mathfrak{m}_{1,1}(i^{*}\alpha_{k})\left|0\right\rangle,$$
for $7\leq k<l\leq 14$;
\item[d)]
$$\mathfrak{q}_{1}(\alpha_{j})\mathfrak{q}_{1}(i^{*}\alpha_{l})\left|0\right\rangle,$$
for $l\in\left\{7,...,14\right\}$;
\item[e)]
$$\mathfrak{q}_{1}(\alpha_{k})\mathfrak{q}_{1}(\alpha_{l})\left|0\right\rangle+\mathfrak{q}_{1}(i^{*}\alpha_{k})\mathfrak{q}_{1}(i^{*}\alpha_{l})\left|0\right\rangle),$$
for $k\in\left\{8,...,14\right\}$ and $l\in\left\{15,...,7+k\right\}$;
\item[f)]
$$\mathfrak{q}_{1}(1)\mathfrak{q}_{1}(x)\left|0\right\rangle.$$
\end{itemize}
\item[2)]
$\rk H^{4}(S^{[2]},\Z)^{\iota}=156$.
\item[3)]
$\discr H^{4}(S^{[2]},\Z)^{\iota}=2^{120}$.
\end{itemize}
\end{lemme}
\textbf{Remark}: 
\begin{itemize}
\item[(a)] The elements of type (a) are products of elements of $U^{3}\oplus(-2)$.
\item[(b)] The elements of type (b) are products of one element of $U^{3}\oplus(-2)$ and one element of $E_{8}(-2)$.
\item[(c)] The elements of type (c) are sums  $x\cdot y+ \iota^{*}(x)\cdot\iota^{*}(y)$ with $x$ and $y$ in $E_{8}(-1)$.
\item[(d)] The elements of type (d) are products of one element of $E_{8}(-1)$ with its image by $\iota^{*}$.
\item[(e)] The elements of type (e) are sums $x\cdot y+ \iota^{*}(x)\cdot\iota^{*}(y)$ with $x$ in $E_{8}(-1)$ and $y$ in $\iota^{*}(E_{8}(-1))$, $y\neq \iota^{*}(x)$.
\end{itemize}
\begin{proof}
\begin{itemize}
\item[1)]
Let $x\in H^{4}(S^{[2]},\Z)^{\iota}$. As follows from Theorem \ref{baseS}, by the same method as in the proof of Lemma \ref{calcul}, we can write:
\begin{align*}
x&=\sum_{1\leq k< j\leq6}{\lambda_{k,j}(\mathfrak{q}_{1}(\alpha_{k})\mathfrak{q}_{1}(\alpha_{j})\left|0\right\rangle)}\\
&+\sum_{1\leq k\leq 6}{\eta_{k}(\mathfrak{q}_{2}(\alpha_{k})\left|0\right\rangle)+\nu_{k}(\mathfrak{m}_{1,1}(\alpha_{k})\left|0\right\rangle)}\\
&+\sum_{1\leq k\leq 6<j\leq 14}{\lambda_{k,j}(\mathfrak{q}_{1}(\alpha_{k})\mathfrak{q}_{1}(\alpha_{j})\left|0\right\rangle+\mathfrak{q}_{1}(\alpha_{k})\mathfrak{q}_{1}(i^{*}\alpha_{j})\left|0\right\rangle)}\\
&+\sum_{7\leq k\leq 14}{\eta_{k}(\mathfrak{q}_{2}(\alpha_{k})\left|0\right\rangle+\mathfrak{q}_{2}(i^{*}\alpha_{k})\left|0\right\rangle)}\\
&+\sum_{7\leq k\leq j\leq14}{\lambda_{k,j}(\mathfrak{q}_{1}(\alpha_{k})\mathfrak{q}_{1}(\alpha_{j})\left|0\right\rangle+\mathfrak{q}_{1}(i^{*}\alpha_{k})\mathfrak{q}_{1}(i^{*}\alpha_{j})\left|0\right\rangle)}\\
&+\sum_{7\leq k\leq 14}{\nu_{k}(\mathfrak{m}_{1,1}(\alpha_{k})\left|0\right\rangle+\mathfrak{m}_{1,1}(i^{*}\alpha_{k})\left|0\right\rangle)}\\
&+\sum_{j=7}^{14}{\lambda_{j,j+8}(\mathfrak{q}_{1}(\alpha_{j})\mathfrak{q}_{1}(i^{*}\alpha_{j})\left|0\right\rangle)}\\
&+\sum_{k=8}^{14}{\sum_{j=15}^{7+k}{\lambda_{k,j}(\mathfrak{q}_{1}(\alpha_{k})\mathfrak{q}_{1}(\alpha_{j})\left|0\right\rangle+\mathfrak{q}_{1}(i^{*}\alpha_{k})\mathfrak{q}_{1}(i^{*}\alpha_{j})\left|0\right\rangle)}}+y\mathfrak{q}_{1}(1)\mathfrak{q}_{1}(x)\left|0\right\rangle,
\end{align*}
with the $\lambda_{k,j}, \mu_{k}, \nu_{k}$ in $\Z$.
\item[2)]
We count the elements of all the types. There are:
\begin{itemize}
\item[]
$\frac{6\times5}{2}+6+6=27$ elements of type a),
\item[]
$6\times8+8=56$ elements of type b),
\item[]
$\frac{8\times7}{2}+8=36$ elements of type c),
\item[]
8 elements of type d),
\item[]
$\frac{8\times8-8}{2}=28$ elements of type e), and
\item[]
1 element of type f).
\end{itemize}
Hence $\rk H^{4}(S^{[2]},\Z)^{\iota}=27+56+36+8+28+1=156$.
\item[3)]
We will show that the discriminant group of $H^{4}(S^{[2]},\Z)^{\iota}$ in $H^{4}(S^{[2]},\Z)$: $A_{H^{4}(S^{[2]},\Z)^{\iota}}$ is $(\Z/2\Z)^{120}$.
First, by Lemma 5.3 of \cite{SmithTh} we have $A_{H^{4}(S^{[2]},\Z)^{\iota}}=(\Z/2\Z)^{a}$, $a\in \mathbb{N}$.
Now, let $x\in A_{H^{4}(S^{[2]},\Z)^{\iota}}$, then there are $y\in H^{4}(S^{[2]},$ $\Z)^{\iota}$ and $z\in H^{4}(S^{[2]},\Z)^{\iota\bot}$ such that
$\overline{\frac{y}{2}}=x$ and $\frac{y+z}{2}\in H^{4}(S^{[2]},\Z)$.
We can write 
\begin{align*}
y&=\sum_{1\leq k< j\leq6}{\lambda_{k,j}(\mathfrak{q}_{1}(\alpha_{k})\mathfrak{q}_{1}(\alpha_{j})\left|0\right\rangle)}\\
&+\sum_{1\leq k\leq 6}{\eta_{k}(\mathfrak{q}_{2}(\alpha_{k})\left|0\right\rangle)+\nu_{k}(\mathfrak{m}_{1,1}(\alpha_{k})\left|0\right\rangle)}\\
&+\sum_{1\leq k\leq 6<j\leq 14}{\lambda_{k,j}(\mathfrak{q}_{1}(\alpha_{k})\mathfrak{q}_{1}(\alpha_{j})\left|0\right\rangle+\mathfrak{q}_{1}(\alpha_{k})\mathfrak{q}_{1}(i^{*}\alpha_{j})\left|0\right\rangle)}\\
&+\sum_{7\leq k\leq 14}{\eta_{k}(\mathfrak{q}_{2}(\alpha_{k})\left|0\right\rangle+\mathfrak{q}_{2}(i^{*}\alpha_{k})\left|0\right\rangle)}\\
&+\sum_{7\leq k\leq j\leq14}{\lambda_{k,j}(\mathfrak{q}_{1}(\alpha_{k})\mathfrak{q}_{1}(\alpha_{j})\left|0\right\rangle+\mathfrak{q}_{1}(i^{*}\alpha_{k})\mathfrak{q}_{1}(i^{*}\alpha_{j})\left|0\right\rangle)}\\
&+\sum_{7\leq k\leq 14}{\nu_{k}(\mathfrak{m}_{1,1}(\alpha_{k})\left|0\right\rangle+\mathfrak{m}_{1,1}(i^{*}\alpha_{k})\left|0\right\rangle)}\\
&+\sum_{j=7}^{14}{\lambda_{j,j+8}(\mathfrak{q}_{1}(\alpha_{j})\mathfrak{q}_{1}(i^{*}\alpha_{j})\left|0\right\rangle)}\\
&+\sum_{k=8}^{14}{\sum_{j=15}^{7+k}{\lambda_{k,j}(\mathfrak{q}_{1}(\alpha_{k})\mathfrak{q}_{1}(\alpha_{j})\left|0\right\rangle+\mathfrak{q}_{1}(i^{*}\alpha_{k})\mathfrak{q}_{1}(i^{*}\alpha_{j})\left|0\right\rangle)}}+y\mathfrak{q}_{1}(1)\mathfrak{q}_{1}(x)\left|0\right\rangle,
\end{align*}
and 
\begin{align*}
z&=\sum_{1\leq k\leq 6<j\leq 14}{\lambda_{k,j}'(\mathfrak{q}_{1}(\alpha_{k})\mathfrak{q}_{1}(\alpha_{j})\left|0\right\rangle-\mathfrak{q}_{1}(\alpha_{k})\mathfrak{q}_{1}(i^{*}\alpha_{j})\left|0\right\rangle)}\\
&+\sum_{7\leq k\leq 14}{\eta_{k}'(\mathfrak{q}_{2}(\alpha_{k})\left|0\right\rangle-\mathfrak{q}_{2}(i^{*}\alpha_{k})\left|0\right\rangle)}\\
&+\sum_{7\leq k\leq j\leq14}{\lambda_{k,j}'(\mathfrak{q}_{1}(\alpha_{k})\mathfrak{q}_{1}(\alpha_{j})\left|0\right\rangle-\mathfrak{q}_{1}(i^{*}\alpha_{k})\mathfrak{q}_{1}(i^{*}\alpha_{j})\left|0\right\rangle)}\\
&+\sum_{7\leq k\leq 14}{\nu_{k}'(\mathfrak{m}_{1,1}(\alpha_{k})\left|0\right\rangle-\mathfrak{m}_{1,1}(i^{*}\alpha_{k})\left|0\right\rangle)}\\
&+\sum_{k=8}^{14}{\sum_{j=15}^{7+k}{\lambda_{k,j}'(\mathfrak{q}_{1}(\alpha_{k})\mathfrak{q}_{1}(\alpha_{j})\left|0\right\rangle-\mathfrak{q}_{1}(i^{*}\alpha_{k})\mathfrak{q}_{1}(i^{*}\alpha_{j})\left|0\right\rangle)}},
\end{align*}
By summing the two expressions, we see that all the coefficients of $y$ in front of elements of type a), d) and f) are even.
Then 
\begin{align*}
\overline{\frac{y}{2}}
&=\sum_{1\leq k\leq 6<j\leq 14}{\overline{\lambda_{k,j}}\left(\overline{\frac{\mathfrak{q}_{1}(\alpha_{k})\mathfrak{q}_{1}(\alpha_{j})\left|0\right\rangle+\mathfrak{q}_{1}(\alpha_{k})\mathfrak{q}_{1}(i^{*}\alpha_{j})\left|0\right\rangle}{2}}\right)}\\
&+\sum_{7\leq k\leq 14}{\overline{\eta_{k}}\left(\overline{\frac{\mathfrak{q}_{2}(\alpha_{k})\left|0\right\rangle+\mathfrak{q}_{2}(i^{*}\alpha_{k})\left|0\right\rangle}{2}}\right)}\\
&+\sum_{7\leq k\leq j\leq14}{\overline{\lambda_{k,j}}\left(\overline{\frac{\mathfrak{q}_{1}(\alpha_{k})\mathfrak{q}_{1}(\alpha_{j})\left|0\right\rangle+\mathfrak{q}_{1}(i^{*}\alpha_{k})\mathfrak{q}_{1}(i^{*}\alpha_{j})\left|0\right\rangle}{2}}\right)}\\
&+\sum_{7\leq k\leq 14}{\overline{\nu_{k}}\left(\overline{\frac{\mathfrak{m}_{1,1}(\alpha_{k})\left|0\right\rangle+\mathfrak{m}_{1,1}(i^{*}\alpha_{k})\left|0\right\rangle}{2}}\right)}\\
&+\sum_{k=8}^{14}{\sum_{j=15}^{7+k}{\overline{\lambda_{k,j}}\left(\overline{\frac{\mathfrak{q}_{1}(\alpha_{k})\mathfrak{q}_{1}(\alpha_{j})\left|0\right\rangle+\mathfrak{q}_{1}(i^{*}\alpha_{k})\mathfrak{q}_{1}(i^{*}\alpha_{j})\left|0\right\rangle}{2}}\right)}}.
\end{align*}
\end{itemize}
The result follows.
\end{proof}
\subsubsection{The lattice $K$}\label{K}
By Lemma \ref{inter}, we get that $\discr K=2^{120}\times2^{156}=2^{276}$.
But all the elements $\pi_{2*}(s^{*}(x))$ for $x$ of type b), c) or e) are divisible by 2, for example: $$\pi_{2*}(s^{*}(\mathfrak{q}_{1}(\alpha_{k})\mathfrak{q}_{1}(\alpha_{j})\left|0\right\rangle+\mathfrak{q}_{1}(i^{*}\alpha_{k})\mathfrak{q}_{1}(i^{*}\alpha_{j})\left|0\right\rangle))=2\pi_{2*}(s^{*}(\mathfrak{q}_{1}(\alpha_{k})\mathfrak{q}_{1}(\alpha_{j})\left|0\right\rangle).$$
Let $\widetilde{K}$ be the overlattice of $K$ where we have divided by 2 all the elements of this kind.
We have added 120 independent halves of elements of $K$ hence $\discr \widetilde{K}=\frac{\discr K}{2^{2\times120}}=\frac{2^{276}}{2^{240}}=2^{36}$.
Now, we will show that $\widetilde{K}$ is primitive in $H^{4}(\widetilde{M},\Z)$.
\subsubsection{The lattice $\pi_{2*}(T)$}
As in Subsection \ref{end1}, we are going to deduce some information on $\pi_{2*}(T)$ from the relation between $H^{4}(\widetilde{M},\Z)$ and $H^{4}(U,\Z)$ .
As in Subsection \ref{UY}, we begin  with the following lemma.
\begin{lemme}
Let $R:=\left\langle (\widetilde{h_{k}})_{k\in\left\{1,...,28\right\}}\cup(\widetilde{\theta_{k}})_{k\in\left\{1,...,22\right\}}\cup\left\{\pi_{2*}(s^{*}(\Sigma))\right\}\right\rangle$. Then
$$H^{4}(U,\Z)\simeq H^{4}(\widetilde{M},\Z)/R.$$
\end{lemme}
\begin{proof}
Consider the following exact sequence:
$$\xymatrix@C=15pt{H^{3}(U,\Z)\ar[r] &H^{4}(\widetilde{M},U,\Z)\ar[r]&H^{4}(\widetilde{M},\Z) \ar[r]&H^{4}(U,\Z) \ar[r]&H^{5}(\widetilde{M},U,\Z)
}.$$
By Proposition \ref{H3}, $H^{3}(U,\Z)=0$. Moreover, by Thom's isomorphism (see Section 11.1.2 of \cite{Voisin}), we have 
$$H^{5}(\widetilde{M},U,\Z)\simeq H^{3}(\widetilde{\Sigma},\Z)\oplus(\oplus_{i=1}^{28}H^{3}(D_{i},\Z)).$$
The map $\widetilde{\Sigma}\rightarrow\Sigma$ is a $\mathbb{P}^{1}$-bundle, hence by the Leray-Serre spectal sequence, we get $H^{3}(\widetilde{\Sigma},\Z)=0$. Moreover $D_{i}\simeq \mathbb{P}^{3}$, so $H^{3}(D_{i},\Z)=0$.
Hence, we get the following exact sequence:
\begin{equation}
\xymatrix@C=15pt{0\ar[r] &H^{4}(\widetilde{M},U,\Z)\ar[r]^{g}&H^{4}(\widetilde{M},\Z) \ar[r]&H^{4}(U,\Z) \ar[r]&0
}.
\label{cohomology}
\end{equation}
We have also $H^{4}(\widetilde{M},U,\Z)\simeq H^{2}(\widetilde{\Sigma},\Z)\oplus(\oplus_{i=1}^{28}H^{2}(D_{i},\Z))$.
Moreover, by the Leray-Serre spectral sequence, we have $H^{2}(\widetilde{\Sigma},\Z)=H^{2}(\Sigma,\Z)\oplus H^{0}(\Sigma,\Z)$.

Look at the following commutative diagram,
$$\xymatrix@C=10pt{0\ar[r] &\ar[d]^{d\pi_{2}^{*}}H^{4}(\mathscr{N}_{\widetilde{M}/F},\mathscr{N}_{\widetilde{M}/F}-0,\Z)=H^{4}(\widetilde{M},U,\Z)\ar[r]^{\ \ \ \ \ \ \ \ \ \ \ \ \ \ \ \ \ \ \ \ g}&\ar[d]_{\pi_{2}^{*}}H^{4}(\widetilde{M},\Z) \ar[r]&\ar[d]H^{4}(U,\Z) \ar[r]&0\\
0\ar[r] &H^{4}(\mathscr{N}_{N_{2}/F},\mathscr{N}_{N_{2}/F}-0,\Z)=H^{4}(N_{2},V,\Z)\ar[r]&H^{4}(N_{2},\Z) \ar[r]&H^{4}(V,\Z) \ar[r]&0,
}$$ 
where $F:=\Fix \iota_{2}$ and $\mathscr{N}_{N_{2}/F}-0$, $\mathscr{N}_{\widetilde{M}/F}-0$ are the vector bundles minus the zero section.
By property of the Thom isomorphism, $d\pi_{2}^{*}(H^{4}(\mathscr{N}_{\widetilde{M}/F},\mathscr{N}_{\widetilde{M}/F}-0,\Z))=2H^{4}(\mathscr{N}_{N_{2}/F},\mathscr{N}_{N_{2}/F}-0,\Z)$.
Then by commutativity of the diagram, we have
$g(H^{4}(\widetilde{M},U,\Z))=R$.
Hence, it follows by \ref{cohomology} that $$H^{4}(U,\Z)\simeq H^{4}(\widetilde{M},\Z)/R.$$
\end{proof}
Now, we can prove the following lemma.
\begin{lemme}
Let $\widetilde{T}$ be the minimal primitive overlattice of $\pi_{2*}(T)$ in $H^{4}(\widetilde{M},\Z)$.
We have $$\discr \widetilde{T}=2^{36+2\rktor H^{4}(\widetilde{M},\Z)},$$
where $\rktor H^{4}(\widetilde{M},\Z)$ is the dimension of the torsion part of $H^{4}(\widetilde{M},\Z)$ as a $\mathbb{F}_{2}$-vector space.
\end{lemme}
\begin{proof}
Let $\widetilde{R}$ be the minimal primitive overlattice of $R$ in $H^{4}(\widetilde{M},\Z)$. By Proposition \ref{torsion} and (\ref{cohomology}), $$\widetilde{R}/R=(\Z/2\Z)^{8-\rktor H^{4}(\widetilde{M},\Z)}.$$
But we already know one element divisible by 2 in $R$.
Indeed, we have $\pi_{2*}(\mathcal{O}_{N_{2}})=\mathcal{O}_{\widetilde{M}}\oplus\mathcal{L}$, with $\mathcal{L}^{2}=\mathcal{O}_{\widetilde{M}}(-D_{1}-...-D_{28}-\widetilde{\Sigma})$.
Thus $$\frac{D_{1}+...+D_{28}+\widetilde{\Sigma}}{2}\in H^{2}(\widetilde{M},\Z).$$
It follows that $$\left(\frac{D_{1}+...+D_{28}+\widetilde{\Sigma}}{2}\right)^{2}=-\frac{\widetilde{h_{1}}+...+\widetilde{h_{28}}+\pi_{2*}(s^{*}(\Sigma))}{
2}\in H^{4}(\widetilde{M},\Z).$$
We are going to deduce $\widetilde{T}/\pi_{2*}(T)$.


If $\pi_{2*}(s^{*}(\Sigma))$ is divisible by 2 in $H^{4}(\widetilde{M},\Z)$, then  $\pi_{2*}(s^{*}(\Sigma))\in (\widetilde{R}/R)\setminus(\widetilde{T}/\pi_{2*}(T))$, if not $\frac{\widetilde{h_{1}}+...+\widetilde{h_{28}}+\pi_{2*}(s^{*}(\Sigma))}{
2}\in (\widetilde{R}/R)\setminus(\widetilde{T}/\pi_{2*}(T))$. 
Then in both cases $$\widetilde{T}/\pi_{2*}(T)=(\Z/2\Z)^{7-\rktor H^{4}(\widetilde{M},\Z)}.$$


Now we can calculate the discriminant of $\widetilde{T}$.
We have $$\discr \widetilde{T}=\frac{\discr \pi_{2*}(T)}{2^{2\times(7-\rktor H^{4}(\widetilde{M},\Z))}}.$$
Moreover by Lemma \ref{inter}, $\discr \pi_{2*}(T)=2^{\rk T}=2^{28+22}=2^{50}$.
Hence $$\discr\widetilde{T}=\frac{2^{50}}{2^{2\times(7-\rktor H^{4}(\widetilde{M},\Z))}}=2^{36+2\rktor H^{4}(\widetilde{M},\Z)}.$$
\end{proof}
\subsubsection{The key Lemma}
\begin{lemme}\label{key0}
The Lattice $\widetilde{K}$ is primitive in $H^{4}(\widetilde{M},\Z)$.
\end{lemme}
\begin{proof}
Let $\mathcal{K}$ be the minimal primitive overlattice of $\widetilde{K}$ in $H^{4}(\widetilde{M},\Z)$.
We have $\widetilde{T}=\mathcal{K}^{\bot}$. Since $H^{4}(\widetilde{M},\Z)$ is unimodular, we have $A_{\mathcal{K}}\simeq A_{\widetilde{T}}$.
In particular, $\discr \mathcal{K}=\discr \widetilde{T}$, so that $\discr \mathcal{K}=2^{36+2\rktor H^{4}(\widetilde{M},\Z)}\geq\discr \widetilde{K}$. Hence, necessarily $\rktor H^{4}(\widetilde{M},\Z)=0$ and $\mathcal{K}=\widetilde{K}$, so that $\widetilde{K}$ is primitive in $H^{4}(\widetilde{M},\Z)$.
\end{proof}
Moreover by the proof follows the following lemma.
\begin{lemme}
The group $H^{4}(\widetilde{M},\Z)$ is without torsion.
\end{lemme}
\subsection{The end of the proof of Theorem \ref{theorem}}
\subsubsection{The sublattice $\pi_{1*}(s_{1}^{*}(U^{3}\oplus (-2)))$ in $H^{2}(M',\Z)$}
In agreement with Section \ref{Focus}, we need the following lemma to finish the proof.
\begin{lemme}\label{key}
The subgroup $\pi_{1*}(s_{1}^{*}(U^{3}\oplus (-2)))$ is primitive in $H^{2}(M',\Z)$.
\end{lemme}
\begin{proof}
We begin with another lemma, and we will end the proof by the same method as in Section \ref{proof}.
\begin{lemme}\label{divisible}
The subgroup $\pi_{2*}(s^{*}(U^{3}\oplus(-2)))$ is primitive in $H^{2}(\widetilde{M},\Z)$.
\end{lemme} 
\begin{proof}
Let $\mathcal{H}$ be the primitive overgroup of $\pi_{2*}(s^{*}(U^{3}\oplus(-2)))$ in $H^{2}(\widetilde{M},\Z)$ and let $x\in \mathcal{H}$.
We can write $$x=\frac{\sum_{(k,j)\in\left\{1,2,3\right\}\times\left\{1,2\right\}}{a_{k,j}\pi_{2*}(s^{*}(u_{k,j}))}+y(\pi_{2*}(s^{*}(\delta)))}{2},$$
where $a_{k,j}$ and $y$ are integers.
Then by (3) of Lemma \ref{inter},
\begin{align*}
x^{2}&=\frac{2\pi_{2*}\left[\left(\sum_{(k,j)\in\left\{1,2,3\right\}\times\left\{1,2\right\}}{a_{k,j}s^{*}(u_{k,j})}+ys^{*}(\delta)\right)^{2}\right]}{4}\\
&=\frac{\pi_{2*}s^{*}\left[\left(\sum_{(k,j)\in\left\{1,2,3\right\}\times\left\{1,2\right\}}{a_{k,j}u_{k,j}}+y\delta\right)^{2}\right]}{2}\\
&=\frac{\pi_{2*}s^{*}\left[\sum_{(k,j)\in\left\{1,2,3\right\}\times\left\{1,2\right\}}{a_{k,j}^{2}u_{k,j}^{2}}+y^{2}\delta^{2}\right]}{2}+A
\end{align*}
with $A\in  H^{4}(\widetilde{M},\Z)$.

We also have by Proposition \ref{base2}:
$$u_{k,j}^{2}=\mathfrak{q}_{1}(v_{k,j})^{2}\left|0\right\rangle=2\mathfrak{m}_{1,1}(v_{k,j})\left|0\right\rangle+\mathfrak{q}_{2}(v_{k,j})\left|0\right\rangle$$
and $$\delta^{2}=\sum_{i<j}{\mu_{i,j}\mathfrak{q}_{1}(\alpha_{i})\mathfrak{q}_{1}(\alpha_{j})\left|0\right\rangle}+\frac{1}{2}\sum_{i}{\mu_{i,i}\mathfrak{q}_{1}(\alpha_{i})^{2}\left|0\right\rangle}+\mathfrak{q}_{1}(1)\mathfrak{q}_{1}(x)\left|0\right\rangle.$$
Then
\begin{align*} x^{2}&=\frac{1}{2}\left[\pi_{2*}s^{*}\left(\sum_{(k,j)\in\left\{1,2,3\right\}\times\left\{1,2\right\}}{a_{k,j}^{2}\mathfrak{q}_{2}(v_{k,j})\left|0\right\rangle}\right.\right.\\
&\left.\left.+y^{2}\left(\sum_{i<j}{\mu_{i,j}\mathfrak{q}_{1}(\alpha_{i})\mathfrak{q}_{1}(\alpha_{j})\left|0\right\rangle}+\frac{1}{2}\sum_{i}{\mu_{i,i}\mathfrak{q}_{2}(\alpha_{i})\left|0\right\rangle}\right)+y^{2}\mathfrak{q}_{1}(1)\mathfrak{q}_{1}(x)\left|0\right\rangle\right)\right]\\
&+A'
\end{align*}
with $A'\in H^{4}(\widetilde{M},\Z)$.

By Lemma \ref{key0}, $x^{2}\in \widetilde{K}$. 
Then by Section \ref{K}, the coefficient of $\mathfrak{q}_{1}(1)\mathfrak{q}_{1}(x)\left|0\right\rangle$ must be integer.
So $y$ is even.

Now, we have 
$$x^{2}=\frac{\pi_{2*}s^{*}\left[\sum_{(k,j)\in\left\{1,2,3\right\}\times\left\{1,2\right\}}{a_{k,j}^{2}\mathfrak{q}_{2}(v_{k,j})\left|0\right\rangle}\right]}{2}+A''$$
with $A''\in H^{4}(\widetilde{M},\Z)$.
Again by Section \ref{K}, all the coefficients $a_{k,j}$ must be even.
We get $x\in\pi_{2*}(s^{*}(U^{3}\oplus(-2)))$.
\end{proof}
The end of the proof is the same as in Section \ref{proof}.
We have the following commutative diagram:
$$\xymatrix{
 \widetilde{M}\ar[r]^{\widetilde{r}} & M'\\
  N_{2} \ar[r]^{s_{2}} \ar[u]^{\pi_{2}}& N_{1} \ar[u]^{\pi_{1}}.
}
\label{diagram}$$
It induces a diagram of maps of cohomology groups:
$$\xymatrix{\mathcal{S} \incl[d] & s_{1}^{*}(U^{3}\oplus(-2)) \incl[d]\\
H^{2}(M',\Z)\ar@/_/[r]_{\pi_{1}^{*}} \ar[d]_{\widetilde{r}^{*}}& H^{2}(N_{1},\Z)\ar@/_/[l]^{\pi_{1*}}\ar[d]^{s_{2}^{*}}\\
H^{2}(\widetilde{M},\Z)\ar@/^/[r]^{\pi_{2}^{*}}& \ar@/^/[l]_{\pi_{2*}} H^{2}(N_{2},\Z)\\
\pi_{2*}(s^{*}(U^{3}\oplus(-2)))\incl[u]& s^{*}(U^{3}\oplus(-2)),\incl[u]}$$
where 
$\mathcal{S}$ is the minimal primitive overgroup of $\pi_{1*}(s_{1}^{*}(U^{3}\oplus (-2)))$ in $H^{2}(M',\Z)$. We know by Lemma \ref{divisible} that $\pi_{2*}(s^{*}(U^{3}\oplus(-2)))$ is primitive in $H^{2}(\widetilde{M},\Z)$.
We are going to show that $\mathcal{S}=\pi_{1*}(s_{1}^{*}(U^{3}\oplus (-2)))$.

Let $x\in \mathcal{S}$, we have $\pi_{1}^{*}(x)\in s_{1}^{*}(U^{3}\oplus(-2))$ and $s_{2}^{*}(\pi_{1}^{*}(x))\in s^{*}(U^{3}\oplus(-2))$.
Then by commutativity, $\pi_{2}^{*}(\widetilde{r}^{*}(x))\in s^{*}(U^{3}\oplus(-2))$, so $2\widetilde{r}^{*}(x)\in\pi_{2*}(s^{*}(U^{3}\oplus(-2)))$. Since $\pi_{2*}(s^{*}(U^{3}\oplus(-2)))$ is primitive in  $H^{2}(\widetilde{M},\Z)$, we have $\widetilde{r}^{*}(x)\in\pi_{2*}(s^{*}(U^{3}\oplus(-2)))$. It means that $\widetilde{r}^{*}(x)=\pi_{2*}(s^{*}(u))$ where $u\in U^{3}\oplus(-2)$.
Then, we have, $\pi_{2}^{*}(\widetilde{r}^{*}(x))=2s^{*}(u)$ and by commutativity, we get $s_{2}^{*}(\pi_{1}^{*}(x))=2s^{*}(u)$. Therefore $\pi_{1}^{*}(x)=2s_{1}^{*}(u)$ and $2x=\pi_{1*}(\pi_{1}^{*}(x))=2\pi_{1*}(s_{1}^{*}(u))$, so $x=\pi_{*}(s^{1}(u))\in \pi_{1*}(s^{*}(U^{3}\oplus(-2)))$.   
\end{proof}
\subsubsection{The element $\overline{\Sigma}'$}
We will now show that the class of $\overline{\Sigma}'$ is not divisible by 2 in $H^{2}(M',\Z)$.
\begin{lemme}\label{sigma2}
The element $\overline{\Sigma}'$ is not divisible by 2 in $H^{2}(M',\Z)$.
\end{lemme}
\begin{proof}
Since $\widetilde{\Sigma}=\widetilde{r}^{*}(\overline{\Sigma}')$, it is enough to show that $\widetilde{\Sigma}$ is not divisible by 2 in $H^{2}(\widetilde{M},\Z)$. We use the same idea as in Section \ref{proof}.
We have the following exact sequence:
$$\xymatrix@C=15pt{H^{1}(U,\Z)\ar[r] &H^{2}(\widetilde{M},U,\Z)\ar[r]&H^{2}(\widetilde{M},\Z) \ar[r]&H^{2}(U,\Z) \ar[r]&H^{3}(\widetilde{M},U,\Z)
}.$$
By the universal coefficient theorem, $H^{1}(U,\Z)$ is torsion free. Since $H^{1}(V,\Z)$ $=0$, we have $H^{1}(U,\Z)=0$.
Moreover, by Thom's isomorphism, we have $H^{3}(\widetilde{M},U,\Z)$ $\simeq H^{1}(\widetilde{\Sigma},\Z)\oplus(\oplus_{i=1}^{28}H^{1}(D_{i},\Z))=0$.
We have also $H^{2}(\widetilde{M},U,\Z)$ $\simeq H^{0}(\widetilde{\Sigma},\Z)\oplus(\oplus_{i=1}^{28}H^{0}(D_{i},\Z))$.
Then the exact sequence gives:
$$H^{2}(U,\Z)\simeq H^{2}(\widetilde{M},\Z)/\left\langle \widetilde{\Sigma},D_{1},...,D_{28}\right\rangle.$$
But by Proposition \ref{H2}, the torsion of $H^{2}(U,\Z)$ is equal to $\Z/2\Z$. This means that if $\widetilde{\mathcal{D}}$ is the minimal primitive overgroup of $\left\langle \widetilde{\Sigma},D_{1},...,D_{28}\right\rangle$ in $H^{2}(\widetilde{M},\Z)$, then $\widetilde{\mathcal{D}}/\left\langle \widetilde{\Sigma},D_{1},...,D_{28}\right\rangle=\Z/2\Z$.
But, we still know that $\widetilde{\Sigma}+D_{1}+...+D_{28}$ is divisible by $2$. So $\widetilde{\mathcal{D}}=\left\langle \widetilde{\Sigma},D_{1},...,D_{28},\frac{\widetilde{\Sigma}+D_{1}+...+D_{28}}{2}\right\rangle$.
Hence $\widetilde{\Sigma}$ is not divisible by 2.
\end{proof}

It remains to find an answer to one more divisibility question.
We know that $\frac{1}{2}\pi_{1*}(s_{1}^{*}(E_{8}(-2)))\oplus\pi_{1*}(s_{1}^{*}(U^{3}))$ $\oplus\Z\pi_{1*}(s_{1}^{*}(\delta))$ and $\Z\overline{\Sigma}'$ are primitive in $H^{2}(M',\Z)$, but it might turn out that an element of the form
$$x=\frac{y\pm \overline{\Sigma}'}{2},$$
with $y\in\frac{1}{2}\pi_{1*}(s_{1}^{*}(E_{8}(-2)))\oplus\pi_{1*}(s_{1}^{*}(U^{3}))\oplus\Z\pi_{1*}(s_{1}^{*}(\delta))$ is integer.
To simplify the formulas, we will use the following notation:
$$\widetilde{\delta}:=\pi_{2*}(s^{*}(\delta)),\ \ \ \ \widetilde{\delta^{2}}:=\pi_{2*}(s^{*}(\delta^{2})),\ \ \ \ \overline{\delta}':=\pi_{1*}(s_{1}^{*}(\delta)),$$
$$\widetilde{u_{k,l}}:=\pi_{2*}(s^{*}(u_{k,l})),\ \ \ \ \ \overline{u_{k,l}}':=\pi_{1*}(s_{1}^{*}(u_{k,l})),$$
for all $k\in \left\{1,2,3\right\}$ and $l\in\left\{1,2\right\}$.
We will show that $\frac{\overline{\delta}'+\overline{\Sigma}'}{2}\in H^{2}(M',\Z)$.
We will need several lemmas.
We will show that $$\frac{\widetilde{\delta^{2}}-\pi_{2*}(s^{*}(\Sigma))}{2}=(\frac{\widetilde{\delta}+\widetilde{\Sigma}}{2})^{2}-\widetilde{\sigma_{\delta}}\in H^{4}(\widetilde{M},\Z);$$
next we will deduce that $$\frac{\widetilde{\delta}+\widetilde{\Sigma}}{2}\in H^{2}(\widetilde{M},\Z),$$ and finally we will be able to prove that $$\frac{\overline{\delta}'+\overline{\Sigma}'}{2}\in H^{2}(M',\Z).$$
\begin{lemme}
The element $\widetilde{\delta^{2}}-\pi_{2*}(s^{*}(\Sigma))$ is divisible by 2 in $H^{4}(\widetilde{M},\Z)$.
\end{lemme}
\begin{proof}
We need to show that $\pi_{2*}(s^{*}(\delta^{2}-\Sigma))$ is divisible by 2.
So, look at $\delta^{2}-\Sigma\in H^{4}(S^{[2]},\Z)$.
By (1) of Lemma \ref{invariant}, we can write:
\begin{align*}
\delta^{2}-\Sigma&=\sum_{1\leq k< j\leq6}{\lambda_{k,j}(\mathfrak{q}_{1}(\alpha_{k})\mathfrak{q}_{1}(\alpha_{j})\left|0\right\rangle)}\\
&+\sum_{1\leq k\leq 6}{\eta_{k}(\mathfrak{q}_{2}(\alpha_{k})\left|0\right\rangle)+\nu_{k}(\mathfrak{m}_{1,1}(\alpha_{k})\left|0\right\rangle)}\\
&+\sum_{1\leq k\leq 6<j\leq 14}{\lambda_{k,j}(\mathfrak{q}_{1}(\alpha_{k})\mathfrak{q}_{1}(\alpha_{j})\left|0\right\rangle+\mathfrak{q}_{1}(\alpha_{k})\mathfrak{q}_{1}(i^{*}\alpha_{j})\left|0\right\rangle)}\\
&+\sum_{7\leq k\leq 14}{\eta_{k}(\mathfrak{q}_{2}(\alpha_{k})\left|0\right\rangle+\mathfrak{q}_{2}(i^{*}\alpha_{k})\left|0\right\rangle)}\\
&+\sum_{7\leq k\leq j\leq14}{\lambda_{k,j}(\mathfrak{q}_{1}(\alpha_{k})\mathfrak{q}_{1}(\alpha_{j})\left|0\right\rangle+\mathfrak{q}_{1}(i^{*}\alpha_{k})\mathfrak{q}_{1}(i^{*}\alpha_{j})\left|0\right\rangle)}\\
&+\sum_{7\leq k\leq 14}{\nu_{k}(\mathfrak{m}_{1,1}(\alpha_{k})\left|0\right\rangle+\mathfrak{m}_{1,1}(i^{*}\alpha_{k})\left|0\right\rangle)}\\
&+\sum_{j=7}^{14}{\lambda_{j,j+8}(\mathfrak{q}_{1}(\alpha_{j})\mathfrak{q}_{1}(i^{*}\alpha_{j})\left|0\right\rangle)}\\
&+\sum_{k=8}^{14}{\sum_{j=15}^{7+k}{\lambda_{k,j}(\mathfrak{q}_{1}(\alpha_{k})\mathfrak{q}_{1}(\alpha_{j})\left|0\right\rangle+\mathfrak{q}_{1}(i^{*}\alpha_{k})\mathfrak{q}_{1}(i^{*}\alpha_{j})\left|0\right\rangle)}}+y\mathfrak{q}_{1}(1)\mathfrak{q}_{1}(x)\left|0\right\rangle,
\end{align*}
with $\lambda_{k,j}, \mu_{k}, \nu_{k}$ in $\Z$.
To see that $\pi_{2*}(s^{*}(\delta^{2}))+\pi_{2*}(\Sigma_{2}^{2})$ is divisible by 2, we need to show that the coefficients of the basis elements of type a), d) and f) are even.
Then we rewrite:
\begin{align*}
\delta^{2}-\Sigma&=\sum_{1\leq k< j\leq6}{\lambda_{k,j}(\mathfrak{q}_{1}(\alpha_{k})\mathfrak{q}_{1}(\alpha_{j})\left|0\right\rangle)}+\sum_{1\leq k\leq 6}{\eta_{k}(\mathfrak{q}_{2}(\alpha_{k})\left|0\right\rangle)+\nu_{k}(\mathfrak{m}_{1,1}(\alpha_{k})\left|0\right\rangle)}\\
&+\sum_{j=7}^{14}{\lambda_{j,j+8}(\mathfrak{q}_{1}(\alpha_{j})\mathfrak{q}_{1}(i^{*}\alpha_{j})\left|0\right\rangle)}\\
&+y\mathfrak{q}_{1}(1)\mathfrak{q}_{1}(x)\left|0\right\rangle+Z,
\end{align*}
where $Z$ is a sum of elements of type b), c), e).
Now we re-arrange the sums as follows:
\begin{align*}
\delta^{2}-\Sigma&=\sum_{(i,j)\neq(l,k)\in \left\{1,2,3\right\}\times\left\{1,2\right\}}{\lambda_{i,j,k,l}(\mathfrak{q}_{1}(v_{i,j})\mathfrak{q}_{1}(v_{k,l})\left|0\right\rangle)}\\
&+\sum_{(l,k)\in \left\{1,2,3\right\}\times\left\{1,2\right\} }{\eta_{l,k}(\mathfrak{q}_{2}(v_{k,l})\left|0\right\rangle)+\nu_{l,k}(\mathfrak{m}_{1,1}(v_{k,l})\left|0\right\rangle)}\\
&+\sum_{j=7}^{14}{\lambda_{j,j+8}(\mathfrak{q}_{1}(\alpha_{j})\mathfrak{q}_{1}(i^{*}\alpha_{j})\left|0\right\rangle)}\\
&+y\mathfrak{q}_{1}(1)\mathfrak{q}_{1}(x)\left|0\right\rangle+Z,
\end{align*}
Making the cup product  of the two sides of the equality by $\mathfrak{q}_{1}(1)\mathfrak{q}_{1}(x)\left|0\right\rangle$ and using Propositions \ref{base2} and \ref{sigma}, we obtain
$y=0$.
Now, again by Proposition \ref{base2}, we can rewrite:
\begin{align*}
\delta^{2}-\Sigma&=\sum_{(i,j)\neq(l,k)\in \left\{1,2,3\right\}\times\left\{1,2\right\}}{\lambda_{i,j,k,l}(u_{i,j}\cdot u_{k,l}-\Delta_{i=k}\mathfrak{q}_{1}(1)\mathfrak{q}_{1}(x)\left|0\right\rangle )}\\
&+\sum_{(l,k)\in \left\{1,2,3\right\}\times\left\{1,2\right\} }{\eta_{k,l}(u_{k,l}\cdot\delta)+\nu_{k,l}(\frac{u_{k,l}^{2}-u_{k,l}\cdot\delta}{2})}\\
&+\sum_{j=7}^{14}{\lambda_{j,j+8}(\gamma_{j}\cdot \iota^{*}\gamma_{j})}\\
&+Z,
\end{align*}
where $\Delta_{i=k}$ is the Kronecker symbol.
Now marking the cup product with $u_{k,l}^{2}$ and using Propositions \ref{beauville}, \ref{base2}, we get $\nu_{k,l}=0$ for all $(k,l)\in\left\{1,2,3\right\}\times\left\{1,2\right\}$.
Next, we get $\eta_{k,l}=0$ by taking the cup product with $u_{k,l}\cdot\delta$.
Now we take the cup product with $u_{i,j}\cdot u_{k,l}$, $i\neq k$, and we obtain that all the 
$\lambda_{i,j,k,l}$ with $i\neq k$ vanish.
Then it remains:
$$\delta^{2}-\Sigma=\sum_{i=1}^{3}{\lambda_{i,i,1,2}(u_{i,1}\cdot u_{i,2}-\mathfrak{q}_{1}(1)\mathfrak{q}_{1}(x)\left|0\right\rangle )}+\sum_{j=7}^{14}{\lambda_{j,j+8}(\gamma_{j}\cdot \iota^{*}\gamma_{j})}+Z$$
Now by taking the cup product with the $u_{i,1}\cdot u_{i,2}$ we get the three equations
\begin{align*}
&-4=-\lambda_{2,2,1,2}-\lambda_{3,3,1,2}\\
&-4=-\lambda_{1,1,1,2}-\lambda_{3,3,1,2}\\
&-4=-\lambda_{2,2,1,2}-\lambda_{1,1,1,2}.\\
\end{align*}
Hence $\lambda_{1,1,1,2}=\lambda_{2,2,1,2}=\lambda_{3,3,1,2}=2$.
So $$\delta^{2}-\Sigma=\sum_{j=7}^{14}{\lambda_{j,j+8}(\gamma_{j}\cdot \iota^{*}\gamma_{j})}+Z',$$
with $Z'=Z+2(u_{1,1}\cdot u_{1,2}+u_{2,1}\cdot u_{2,2}+u_{3,1}\cdot u_{3,2}-3\mathfrak{q}_{1}(1)\mathfrak{q}_{1}(x)\left|0\right\rangle)$.

Now, it remains to handle the cup-products $\gamma_{j}\cdot \iota^{*}\gamma_{j}$.
We recall that the lattice $E_{8}$ can be embedded in $\mathbb{R}^{8}$ with its canonical scalar product as the lattice, freely generated by the columns of the matrix
$$\begin{bmatrix}
2 & -1 & 0 &0 & 0 & 0 &0 
& \frac{1}{2} \\
0 & 1 & -1 & 0 &0 & 0 & 0
&  \frac{1}{2}\\
0 & 0 & 1 & -1 & 0 & 0 & 0
& \frac{1}{2} \\
0 & 0 & 0 & 1 & -1& 0 & 0
& \frac{1}{2}\\
0 & 0 & 0 & 0 &1 & -1& 0 
& \frac{1}{2}\\
0 & 0 & 0 & 0 & 0 &1 & -1
& \frac{1}{2}\\
0 & 0 & 0 & 0 & 0 &0 &1 
& \frac{1}{2}\\
0 & 0 & 0 & 0 & 0 &0 &0 
& \frac{1}{2}
\end{bmatrix}.$$
With this identification, we use the columns of this matrix as the basis of $E_{8}(-1)\subset H^{2}(S,\Z)$. 

Then making the cup product of $\delta^{2}-\Sigma$ with $\gamma_{14}\cdot \iota^{*}\gamma_{14}$,
we get: $2=-2\lambda_{14,14+8}-\lambda_{7,7+8}+ \gamma_{14}\cdot \iota^{*}\gamma_{14}\cdot Z'$. Since $\gamma_{14}\cdot \iota^{*}\gamma_{14}\cdot Z'$ is necessarily even, we see that $\lambda_{7,7+8}$ is even.
Next, we take the cup product with $\gamma_{7}\cdot \iota^{*}\gamma_{7}$, and we get that $\lambda_{14,14+8}$ is even;
we go on with the cup products with $\gamma_{13}\cdot \iota^{*}\gamma_{13}$,..., $\gamma_{8}\cdot \iota^{*}\gamma_{8}$, and we get that all the $\lambda_{j,j+8}$ are even.
\end{proof}

We will deduce that $\widetilde{\delta}+\widetilde{\Sigma}$ is divisible by 2 in $H^{2}(\widetilde{M},\Z)$. To this end, we will use Smith theory as in Section \ref{ASmith}.

Look at the following exact sequence:
$$\xymatrix@C=10pt@R0pt{0\ar[r] &H^{2}(\widetilde{M},\widetilde{\Sigma}\cup(\cup_{k=1}^{8}D_{k}),\mathbb{F}_{2}))\ar[r]&H^{2}(\widetilde{M},\mathbb{F}_{2}) \ar[r]& H^{2}(\widetilde{\Sigma}\cup(\cup_{k=1}^{8}D_{k},\mathbb{F}_{2}))\\
\ar[r]&H^{3}(\widetilde{M},\widetilde{\Sigma}\cup(\cup_{k=1}^{28}D_{k}),\mathbb{F}_{2})\ar[r] &0.&
}$$
First, we will calculate the vector spaces $H^{2}(\widetilde{M},\widetilde{\Sigma}\cup(\cup_{k=1}^{28}D_{k}),\mathbb{F}_{2})$ and $H^{3}(\widetilde{M},$ $\widetilde{\Sigma}\cup(\cup_{k=1}^{28}D_{k}),\mathbb{F}_{2})$.
By 3) of Proposition \ref{SmithProp}, we have 
$$H^{*}(\widetilde{M},\widetilde{\Sigma}\cup(\cup_{k=1}^{28}D_{k}),\mathbb{F}_{2})\simeq H^{*}_{\sigma}(N_{2}).$$
\begin{lemme}
We have:$$h^{2}_{\sigma}(N_{2})=36,\ \ \ \ \ h^{3}_{\sigma}(N_{2})=43.$$
\end{lemme}
\begin{proof}
The previous exact sequence gives us the following equation:

$$h^{2}_{\sigma}(N_{2})-h^{2}(\widetilde{M},\mathbb{F}_{2})+h^{2}(\widetilde{\Sigma}\cup(\cup_{k=1}^{28}D_{k}),\mathbb{F}_{2})-h^{3}_{\sigma}(N_{2})=0.$$
As $h^{2}(\widetilde{M},\mathbb{F}_{2})=16+28=44$ and $h^{2}(\widetilde{\Sigma}\cup(\cup_{k=1}^{28}D_{k}),\mathbb{F}_{2})=23+28=51$, we obtain:
$$h^{2}_{\sigma}(N_{2})-h^{3}_{\sigma}(N_{2})=7.$$

Moreover by 2) of Proposition \ref{SmithProp}, we have the exact sequence
$$\xymatrix@C=10pt@R0pt{0\ar[r] &H^{1}_{\sigma}(N_{2})\ar[r]&H^{2}_{\sigma}(N_{2}) \ar[r]&H^{2}(N_{2},\mathbb{F}_{2}) \ar[r]&H^{2}_{\sigma}(N_{2})\oplus H^{2}(\Sigma_{2}\cup(\cup_{k=1}^{28}E_{k}),\mathbb{F}_{2})\\
\ar[r]&H^{3}_{\sigma}(N_{2})\ar[r]&0. & &
}$$
By Lemma 7.4 of \cite{SmithTh}, $h^{1}_{\sigma}(N_{2})=h^{0}(\Sigma_{2}\cup(\cup_{k=1}^{28}E_{k}),\mathbb{F}_{2})-1$.
Then we get the equation
\begin{align*}
&h^{0}(\Sigma_{2}\cup(\cup_{k=1}^{28}E_{k}),\mathbb{F}_{2})-1-h^{2}_{\sigma}(N_{2})+h^{2}(N_{2},\mathbb{F}_{2})\\
&-h^{2}_{\sigma}(N_{2})-h^{2}(\Sigma_{2}\cup(\cup_{k=1}^{28}E_{k}),\mathbb{F}_{2})+h^{3}_{\sigma}(N_{2})=0,\\
\end{align*}
or
$$29-2h^{2}_{\sigma}(\widetilde{X})+h^{3}_{\sigma}(\widetilde{X})=0.$$
From the two equations, we deduce that
$$h^{2}_{\sigma}(N_{2})=36,\ \ \ \ \ h^{3}_{\sigma}(N_{2})=43.$$
\end{proof}
\begin{lemme}\label{seven}
The following seven elements belong to $H^{2}(\widetilde{M},\Z)$:
$\frac{\widetilde{u_{k,l}}+d_{k,l}}{2}$ $(k,l)\in \left\{1,2,3\right\}\times\left\{1,2\right\}$
and $\frac{\widetilde{\delta}+d_{\delta}}{2}$. Moreover, $\Vect_{\mathbb{F}_{2}}((d_{i,j})_{(i,j)\in \left\{1,2,3\right\}\times\left\{1,2\right\}},d_{\delta})$ is a subspace of $\Vect_{\mathbb{F}_{2}}(D_{1},...,D_{28})$ of dimension 7.
\end{lemme}
\begin{proof}
Come back to the exact sequence
$$\xymatrix@C=15pt{0\ar[r] &H^{2}(\widetilde{M},\widetilde{\Sigma}\cup(\cup_{k=1}^{28}D_{k}),\mathbb{F}_{2})\ar[r]&\ar[r]^{\varsigma^{*}\ \ \ \ \ \ \  }H^{2}(\widetilde{M},\mathbb{F}_{2}) & H^{2}(\widetilde{\Sigma}\cup(\cup_{k=1}^{28}D_{k}),\mathbb{F}_{2})
},$$
where $\varsigma:\widetilde{\Sigma}\cup(\cup_{k=1}^{28}D_{k})\hookrightarrow \widetilde{M}$ is the inclusion.
Since $h^{2}(\widetilde{M},\widetilde{\Sigma}\cup(\cup_{k=1}^{28}D_{k}),\mathbb{F}_{2})=h^{2}_{\sigma}(N_{2})=36$, we have $\dim_{\mathbb{F}_{2}} \varsigma^{*}(H^{2}(\widetilde{M},\mathbb{F}_{2}))=(16+28)-36=8$.
We can interpret this as follows.
Consider the homomorphism
\begin{align*}
\varsigma^{*}_{\Z}:H^{2}(\widetilde{M},\Z)&\rightarrow H^{2}(\widetilde{\Sigma},\Z)\oplus (\oplus_{k=1}^{28} H^{2}(D_{k},\Z))\\
 u&\rightarrow (u\cdot\widetilde{\Sigma},u\cdot D_{1},...,u\cdot D_{28}).
\end{align*}
Since this is a map of torsion free $\Z$-modules, we can tensor by $\mathbb{F}_{2}$,
$$\varsigma^{*}=\varsigma^{*}_{\Z}\otimes \id_{\mathbb{F}_{2}}: H^{2}(\widetilde{M},\Z)\otimes\mathbb{F}_{2}\rightarrow H^{2}(\widetilde{\Sigma},\Z)\oplus (\oplus_{k=1}^{28} H^{2}(D_{k},\Z))\otimes\mathbb{F}_{2},$$
and we have 8 independent elements such that the intersection with the $D_{k}$ $k\in\left\{1,...,8\right\}$ and $\widetilde{\Sigma}$ are not all zero.
But, $\varsigma^{*}(\pi_{2*}(s^{*}(U^{3}\oplus(-2)))\oplus\frac{1}{2}\pi_{2*}(s^{*}(E_{8}(-2)))\oplus \left\langle D_{1},...,D_{28},\widetilde{\Sigma}\right\rangle)=0$, (it follows from Lemma \ref{inter} (3)).
Hence, there are 8 more independent elements in $H^{2}(\widetilde{M},$ $\Z)$. Moreover, these elements must be of the form $\frac{u+d}{2}$
with $u\in\pi_{2*}(s^{*}(U^{3}\oplus(-2)))$ and $d\in \left\langle D_{1},...,D_{28},\widetilde{\Sigma}\right\rangle$.
Indeed, applying $\pi_{2*}$ to an element of the form $\frac{e_{k}+u_{k}+n_{k}}{2}$ with $e_{k}\in \frac{1}{2}\pi_{2*}(s^{*}(E_{8}(-1)))$,
we see that $e_{k}$ is divisible by 2.

By Proposition \ref{interr}, we know that $\varsigma^{*}(\frac{\widetilde{\Sigma}+D_{1}+...+D_{28}}{2})\neq 0$.
Let $$x_{1}:=\varsigma^{*}_{\Z}(\frac{\widetilde{\Sigma}+D_{1}+...+D_{28}}{2})\in H^{2}(\widetilde{\Sigma},\Z)\oplus(\oplus_{k=1}^{28} H^{2}(D_{k},\Z))$$ and  $$\overline{x_{1}}:=x_{1}\otimes 1\in H^{2}(\widetilde{\Sigma},\mathbb{F}_{2})\oplus(\oplus_{k=1}^{28} H^{2}(D_{k},\mathbb{F}_{2})).$$
We complete the family $(\overline{x_{1}})$ to a basis $(\overline{x_{k}}:=x_{k}\otimes 1)_{1\leq k\leq8}$ of $\varsigma^{*}(H^{2}(\widetilde{M},\mathbb{F}_{2}))$.
For all $1\leq k\leq8$, we have $$x_{k}=\varsigma^{*}_{\Z}(\frac{u_{k}+d_{k}}{2}),$$ 
where $u_{k}\in\pi_{2*}(s^{*}(U^{3}\oplus(-2)))$ and $d_{k}$ is an integer combination of the $D_{l}$, $1\leq l\leq 28$, and $\widetilde{\Sigma}$. 
In particular, $u_{1}=0$ and $d_{1}=\widetilde{\Sigma}+D_{1}+...+D_{28}$. 

Now, let $\mathcal{D}$ be the vector subspace of $\Vect_{\mathbb{F}_{2}}(\widetilde{\Sigma},D_{1},...,D_{28})$ generated by the $d_{k}$, $1\leq k\leq 8$.
We have $$\dim_{\mathbb{F}_{2}} \mathcal{D}= 8.$$ 
The fact that the family $(\overline{x_{k}}:=x_{k}\otimes 1)_{1\leq k\leq8}$ of $\varsigma^{*}(H^{2}(\widetilde{M},\mathbb{F}_{2}))$ is a basis, will imply that the $d_{k}$ are $\mathbb{F}_{2}$-linearly independent.

To show this, we just need to see that $(d_{k})_{1\leq k\leq8}$ is free.
Assume $\sum_{k=1}^{8}{ \overline{\epsilon_{k}}\overline{d_{k}}}=0$, then $\sum_{k=1}^{8}{ \epsilon_{k} d_{k}}=2d$, where $d$ is an integer combination of the $D_{k}$ and $\widetilde{\Sigma}$ by the definition of $\mathcal{D}$.
Now we have $\frac{\sum_{k=1}^{8}\epsilon_{k}(u_{k}+d_{k})}{2}=\frac{\sum_{k=1}^{8}\epsilon_{k}u_{k}}{2}+d$.
By Lemma \ref{divisible}, $\frac{\sum_{k=1}^{8}\epsilon_{k}u_{k}}{2}\in\pi_{2*}(s^{*}(U^{3}\oplus(-2)))$, so $\varsigma^{*}\left(\frac{\sum_{k=1}^{8}\epsilon_{k}u_{k}}{2}\right)=0$.
Then
\begin{align*}
\sum_{k=1}^{8}{\epsilon_{k}x_{k}}&=\sum_{k=1}^{8}{\epsilon_{k}\varsigma^{*}_{\Z}\left(\frac{u_{k}+d_{k}}{2}\right)}\\
&=\varsigma^{*}_{\Z}\left(\frac{\sum_{k=1}^{8}\epsilon_{k}(u_{k}+d_{k})}{2}\right)\\
&=\varsigma^{*}_{\Z}\left(\frac{\sum_{k=1}^{8}\epsilon_{k}u_{k}}{2}\right)+\varsigma^{*}_{\Z}(d)\\
\end{align*}
and $\sum_{k=1}^{8}{\overline{\epsilon_{k}}\overline{x_{k}}}=\varsigma^{*}(d)=0$.


Now let $\mathcal{U}$ be the subspace of $\Vect_{\mathbb{F}_{2}}((\widetilde{u_{l,m}})_{1\leq l\leq3,1\leq m\leq2},\widetilde{\delta})$ generated by the $u_{k}$, $2\leq k\leq 8$.
We will show that $$\mathcal{U}=\Vect_{\mathbb{F}_{2}}((\widetilde{u_{l,m}})_{1\leq l\leq3,1\leq m\leq2},\widetilde{\delta}).$$
To do this, we just need to show that the family $(u_{k})_{2\leq k\leq 8}$ is free.
We have $\sum_{k=2}^{8}{ \overline{\epsilon_{k}}u_{k}}=0$, hence $\sum_{k=2}^{8}{ \epsilon_{k} u_{k}}=2u$, where $u$ is an integer combination of the $\widetilde{u_{l,m}}$ and $\widetilde{\delta}$ by the definition of $\mathcal{U}$.
Then $d=\frac{\sum_{k=2}^{8}{ \epsilon_{k} d_{k}}}{2}$ is integer.
By the proof of Lemma \ref{sigma2}, there are just two possibilities:
$d\in \left\langle (D_{k})_{k\in\left\{1,...,28\right\}}, \widetilde{\Sigma}\right\rangle$ or $d=\frac{\widetilde{\Sigma}+D_{1}+...+D_{28}}{2}$.
In the first case we get $\sum_{k=2}^{8}{\overline{\epsilon_{k}}\overline{x_{k}}}=\varsigma^{*}(d)=0$, so that the $\epsilon_{k}$ are even for all $k$.
In the second case, we get $\sum_{k=2}^{8}{ \overline{\epsilon_{k}} d_{k}}=d_{1}$, which is impossible.

This proves the existence of elements
$\frac{\widetilde{u_{k,l}}+d_{k,l}}{2}$, $(k,l)\in \left\{1,2,3\right\}\times\left\{1,2\right\}$
and $\frac{\widetilde{\delta}+d_{\delta}}{2}$ in $H^{2}(\widetilde{M},\Z)$, for which $d_{k,l}$ and $d_{\delta}$ integer combination of the $D_{l}$, $1\leq l\leq 28$ and $\widetilde{\Sigma}$. We can suppose that they are only combinations of the $D_{l}$. If this is not the case, we just have to add $\frac{\widetilde{\Sigma}+D_{1}+...+D_{28}}{2}$. And finally, by Lemma \ref{divisible}, $\Vect_{\mathbb{F}_{2}}((d_{i,j})_{(i,j)\in \left\{1,2,3\right\}\times\left\{1,2\right\}},d_{\delta})$ is a subspace of $\Vect_{\mathbb{F}_{2}}(D_{1},...,D_{28})$ of dimension 7.
\end{proof}
\begin{lemme}
The element $\widetilde{\delta}+\widetilde{\Sigma}$ is divisible by 2 in $H^{2}(\widetilde{M},\Z)$.
\end{lemme}
\begin{proof}
Multiplying by $\frac{\widetilde{\Sigma}+D_{1}+...+D_{28}}{2}$ the 7 elements of the Lemma \ref{seven},
we get seven independent elements of the form $\frac{\widetilde{\sigma_{u_{k,l}}}+h_{k,l}}{2}$, $(k,l)\in \left\{1,2,3\right\}\times\left\{1,2\right\}$
and $\frac{\widetilde{\sigma_{\delta}}+h}{2}$; the elements 
$h_{k,l}$, and $h$ are integer combinations of the $(\widetilde{h_{k}})_{k\in\left\{1,...,28\right\}}$.
We know that $\widetilde{T}/\pi_{2*}(T)=(\Z/2\Z)^{7}$, hence $$\widetilde{T}=\left\langle \pi_{2*}(T),(\frac{\widetilde{\sigma_{u_{k,l}}}+h_{k,l}}{2})_{(k,l)\in \left\{1,2,3\right\}\times\left\{1,2\right\}}, \frac{\widetilde{\sigma_{\delta}}+h}{2} \right\rangle.\ \ \ \ \ (*)$$
We have shown that $\frac{\widetilde{\delta}+d_{\delta}}{2}\in H^{2}(\widetilde{M},\Z)$; now we will show that $d_{\delta}=D_{1}+...+D_{28}$. It will follow that $\widetilde{\delta}+\widetilde{\Sigma}$ is divisible by 2.
We can write $d_{\delta}=\epsilon_{1}D_{1}+...+\epsilon_{28}D_{28}$ with $\epsilon_{k}\in \left\{0,1\right\}$.
We have $$\left(\frac{\widetilde{\delta}+d_{\delta}}{2}\right)^{2}=\frac{\widetilde{\delta^{2}}+\sum_{k=1}^{28}{\epsilon_{k}\widetilde{h_{k}}}}{2}.$$
We have also $$\left(\frac{\widetilde{\Sigma}+D_{1}+...+D_{28}}{2}\right)^{2}=\frac{-\pi_{2*}(s^{*}(\Sigma))+\widetilde{h_{1}}+...+\widetilde{h_{28}}}{2}.$$
We sum up these two elements and we get the element $$\frac{\widetilde{\delta^{2}}-\pi_{2*}(s^{*}(\Sigma))+\sum_{k=1}^{28}{(\epsilon_{k}+1)\widetilde{h_{k}}}}{2}.$$
Since $\widetilde{\delta^{2}}-\pi_{2*}(s^{*}(\Sigma))$ is divisible by 2, we see that
$\frac{\sum_{k=1}^{28}{(\epsilon_{k}+1)\widetilde{h_{k}}}}{2}$ is integer. 
Then by (*) and Proposition \ref{interr} (4), the unique possibility is $\epsilon_{k}=1$ for all $k\in\left\{1,...,28\right\}$.
\end{proof}
\begin{lemme}
The element $\overline{\delta}'+\overline{\Sigma}'$ is divisible by 2 in $H^{2}(M',\Z)$.
\end{lemme}
\begin{proof}
We can find a Cartier divisor on $\widetilde{M}$ which corresponds to $\frac{\pi_{2*}(s^{*}(\delta))+\widetilde{\Sigma}}{2}$ and which does not meet 
$\cup_{k=1}^{28} \widetilde{r}^{-1}(p_{k})$.
Then this Cartier divisor induces a Cartier divisor on $M'$ which necessarily corresponds to half the cocycle $\pi_{1*}(s_{1}^{*}(\delta))$ $+\overline{\Sigma}'$.
\end{proof}
Finally, we get the following theorem.
\begin{thm}\label{fin}
We have $$H^{2}(M',\Z)=\frac{1}{2}\pi_{1*}(s_{1}^{*}(E_{8}(-2)))\oplus\pi_{1*}(s_{1}^{*}(U^{3}))\oplus\Z(\frac{\overline{\delta}'+\overline{\Sigma}'}{2})\oplus\Z(\frac{\overline{\delta}'-\overline{\Sigma}'}{2}).$$
\end{thm}
\subsubsection{Last calculations}
Now we are able to finish the calculation of the Beauville--Bogomolov form on $H^{2}(M',\Z)$.
By Propositions \ref{Mong}, \ref{passage}, \ref{delta}, \ref{ortho} and Theorem \ref{fin},
the Beauville--Bogomolov form on $H^{2}(M',\Z)$ gives the lattice:
$$\frac{1}{2}E_{8}\left(-2\sqrt{\frac{24}{C_{M'}}}\right)\oplus U^{3}\left(\sqrt{\frac{24}{C_{M'}}}\right) \oplus \left(-\sqrt{\frac{24}{C_{M'}}}\right)^{2}$$
$$=E_{8}\left(-\sqrt{\frac{6}{C_{M'}}}\right)\oplus U^{3}\left(2\sqrt{\frac{6}{C_{M'}}}\right) \oplus \left(-2\sqrt{\frac{6}{C_{M'}}}\right)^{2}$$
It follows that $C_{M'}=6$, and we get Theorem \ref{theorem}.
\begin{prop}\label{b}
We have:
\begin{itemize}
\item $b_{3}(M')=0,$
\item $b_{4}(M')=178,$
\item $\chi(M')=212.$
\end{itemize}
\end{prop}
\begin{proof}
By \cite{Voisin} Theorem 7.31, we have:
$$H^{3}(N_{1},\Z)=0,\ \ \ \ H^{4}(N_{1},\Z)\simeq H^{4}(S^{[2]},\Z)+H^{2}(\Sigma,\Z).$$
Then we get $b_{3}(M')=0$. 
Since $H^{4}(N_{1},\Z)^{\iota_{1}}\simeq H^{4}(S^{[2]},\Z)^{\iota}+H^{2}(\Sigma,\Z)$, 
by Lemma \ref{invariant}, we have $\rk H^{4}(N_{1},\Z)^{\iota_{1}}=156+22=178$.

By Poincaré duality, we have $\rk H^{5}(N_{1},\Z)^{\iota_{1}}=\rk H^{7}(N_{1},\Z)^{\iota_{1}}=0$, $\rk H^{6}(N_{1},$ $\Z)^{\iota_{1}}=16$ and $\rk H^{8}(N_{1},\Z)^{\iota_{1}}=1$.
Finally $\chi(M')=1-0+16-0+178-0+16-0+1=212$.
\end{proof}
\section{Application to the V-manifold of Markushevich and Tikhomirov}
\subsection{Construction}
In this section we follow \cite{Markou}.

Let $X$ be a dell Pezzo surface of degree 2 obtained as a double cover of $\mathbb{P}^2$ branched in a generic quartic curve $B_{0}$. Let $\eta: X\rightarrow \mathbb{P}^{2}$. Let $D_{0}$ be a generic curve from the linear system $|-2K_{X}|$, $\rho: S\rightarrow X$ the double cover branched in $D_{0}$. Then $S$ is a K3 surface, and $H=\rho^{*}(-K_{X})$ is a degree-4 ample divisor class on $S$. We will denote by $\tau$ the Galois involution of the double cover $\rho$.


Let $\mathcal{M}=M_{S}^{H,s}(0,H,-2)$ the moduli space of semistable sheaves $\mathcal{F}$ on $S$ with respect to the ample class $H$ with Mukai vector $v(\mathcal{F})=(0,H,-2)$. We consider the following involution on $\mathcal{M}$
$$\sigma:\mathcal{M}\rightarrow\mathcal{M}, [\mathcal{L}]\mapsto [\mathcal{E}xt_{\mathcal{O}_{S}}^1(\mathcal{L},\mathcal{O}_{S}(-H))].$$
We set $\kappa=\tau\circ\sigma$.
In \cite{Markou} it is proved that $\kappa$ is a regular involution on $\mathcal{M}$ and that its fixed locus has one 4-dimensional irreducible component plus 64 isolated points.
\begin{defi}
We define $\mathcal{P}$ as the 4-dimensional component of $Fix(\kappa)$.
\end{defi}

\begin{thm}
The variety $\mathcal{P}$ is an irreducible symplectic V-manifolds of dimension 4 with only 28 points of singularity analytically equivalent to 
$(\mathbb{C}^4 / \left\{\pm1\right\}$ $,0)$.
\end{thm}
\begin{proof}
See Theorem 3.4 and Corollary 5.7 of \cite{Markou}.
\end{proof}
In fact, $\mathcal{P}$ is bimeromorphic to a partial resolution of a quotient of $S^{[2]}$.
Consider Beauville's involution (see Section 6 of \cite{Beau}):$$\iota_{0}: S^{[2]}\rightarrow S^{[2]}, \xi\mapsto\xi'=(\left\langle \xi\right\rangle\cap S)-\xi.$$
We consider $S$ as a quartic surface in $\mathbb{P}^3$ via its embedding given by the linear system $|H|$, $\left\langle \xi\right\rangle$ stands for the line in $\mathbb{P}^3$ spanned by $\xi$, and $\xi'$ is the residual intersection of $\left\langle \xi\right\rangle$ with $S$. By \cite{Beau}, this involution is regular whenever $S$ contains no lines, which is the case for sufficiently generic $S$. Moreover, $\tau$ induces on $S^{[2]}$ an involution which we will denote by the same symbol. The composition $\iota=\iota_{0}\circ\tau$ is also an involution because $\tau$ on $S$ is the restriction of a linear involution on $\mathbb{P}^3$ and $\iota_{0}$ commutes with $\tau$.

By construction $\tau$ is an anti-symplectic involution and by Proposition 4.1 of \cite{Grady3}, also $\iota_{0}$ is.
Then by Theorem 4.1 of \cite{Mongardi}, the fixed locus of $\iota$ is the union of a nonsingular irreducible surface $\Sigma\subset S^{[2]}$ and 28 isolated points.
We denote $M=S^{[2]}/\iota$ and $\overline{\Sigma}$ the image of $\Sigma$ in $M$. We also denote by $M'$ the partial resolution of singularities of $M$ obtained by blowing up $\overline{\Sigma}$ and $\overline{\Sigma}'$ the exceptional divisor of the blowup.
\begin{thm}
The variety $M'$ is an irreducible symplectic V-manifold whose singularities are 28 points of analytic type $(\mathbb{C}^4 / \left\{\pm1\right\},0)$.
Moreover there is a Mukai flop between $M'$ and $\mathcal{P}$, which is an isomorphism between $M'\setminus\Pi'$and $\mathcal{P}\setminus\Pi$, where $\Pi'$and $\Pi$ are subvarieties of codimension 2.
\end{thm}
\begin{proof}
See Corollary 5.7 of \cite{Markou}.
\end{proof}
\subsection{The results}
As a consequence of Corollary 5.7 of \cite{Markou} and Theorem \ref{theorem}, we have the following corollary.
\begin{cor}
The Beauville--Bogomolov lattice $H^{2}(\mathcal{P},\Z)$ is isomorphic to $E_{8}(-1)\oplus U(2)^3\oplus(-2)^{2}$.
Moreover, the Fujiki constant $C_{\mathcal{P}}=6$.
\end{cor}
As a consequence of Corollary 5.7 of \cite{Markou}, Proposition \ref{b2} and Proposition \ref{b}, we can state the following proposition.
\begin{prop}
The variety $\mathcal{P}$ has the following numerical invariants: 
$$b_{2}(\mathcal{P})=16,\ 
b_{3}(\mathcal{P})=0,\ 
b_{4}(\mathcal{P})=178,\ 
\chi(\mathcal{P})=212.$$
\end{prop}
\bibliographystyle{amssort}

Departement of Mathematics, Lille1 University, 59655 Villeneuve d'Ascq

E-mail address: \texttt{gregoire.menet@ed.univ-lille1.fr}

\end{document}